\documentclass[11pt,reqno]{amsart}
\usepackage{amssymb,amsmath,amscd,amsthm,amsfonts,wasysym,mathrsfs,amsxtra}
\usepackage{mathtools}
\usepackage{marginnote}
\usepackage[title]{appendix}
\usepackage{hyperref}
\usepackage{comment}
\usepackage[]{mdframed}
\hypersetup{colorlinks,linkcolor=blue,urlcolor=blue,citecolor=blue}
\usepackage[table]{xcolor}
\usepackage{tikz}
\usepackage[all,cmtip]{xy}
\usepackage{tikz-cd}

\newcommand{\CC}{\mathbb{C}}

\newcommand{\FF}{\mathbb{F}}
\newcommand{\NN}{\mathbb{N}}

\newcommand{\Aa}{\mathcal{A}}

\newcommand{\cg}{\mathfrak{g}}

\newcommand{\fS}{\mathfrak{S}}

\newcommand{\fU}{\mathfrak{U}}

\newcommand{\Kk}{\mathcal{K}}

\newcommand{\Nn}{\mathcal{N}}

\newcommand{\Uu}{\mathcal{U}}

\newcommand{\V}{\mathcal{V}}

\newcommand{\ZZ}{\mathbb{Z}}

\newcommand{\E}{\sf{E}}
\newcommand{\F}{\sf{F}}
\newcommand{\TTt}{\sf{T}}
\newcommand{\SSs}{\sf{S}}

\newcommand{\Hhh}{\sf{H}}

\newcommand{\spi}{\sf{\Psi}}
\newcommand{\GL}{\mathfrak{gl}}
\newcommand{\SL}{\mathfrak{sl}}
\newcommand{\kk}{\underline{k}}
\newcommand{\Ll}{\underline{l}}

\newcommand{\bo}{\boldsymbol{1}}

\newcommand{\blam}{\boldsymbol{\lambda}}

\newcommand{\tdim}{\mathrm{dim}}

\newcommand{\Hom}{\mathrm{Hom}}

\newcommand{\Sym}{\mathrm{Sym}}
\newcommand{\Rep}{\mathrm{Rep}}

\newcommand{\Mat}{\mathrm{Mat}}
\newcommand{\Rm}{\mathrm}

\newcommand{\GLL}{\mathrm{GL}}

\newcommand{\Fl}{\mathrm{Fl}}

\newcommand{\Tr}{\mathrm{Tr}}
\newcommand{\MF}{\mathrm{MF}}

\newcommand{\Span}{\mathrm{Span}}

\newcommand{\Coh}{\mathrm{Coh}}

\newcommand{\Inv}{\mathrm{Inv}}
\newcommand{\gr}{\mathrm{gr}}
\newcommand{\D}{\mathrm{D}}

\newcommand{\HA}{\mathrm{HA}}
\newcommand{\CoHA}{\mathrm{CoHA}}
\newcommand{\KHA}{\mathrm{KHA}}
\newcommand{\aff}{\mathrm{aff}}

\newcommand{\brmU} {\boldsymbol{\mathrm{U}}}

\newcommand{\brmS} {\boldsymbol{\mathrm{S}}}
\newcommand{\brmT} {\boldsymbol{\mathrm{T}}}

\newtheorem{theorem}{Theorem}[section]
\newtheorem{thmx}{Theorem}

\newtheorem{lemma}[theorem]{Lemma}

\newtheorem{proposition}[theorem]{Proposition}

\newtheorem{corollary}[theorem]{Corollary}
\theoremstyle{definition}
\newtheorem{definition}[theorem]{Definition}

\newtheorem{example}[theorem]{Example}

\theoremstyle{remark}
\newtheorem{remark}[theorem]{Remark}
\newtheorem{conjecture}[theorem]{Conjecture}

\numberwithin{equation}{section}

\title[$0$-Affine Quantum Group as K-Theoretic Hall Algebra]{$0$-Affine Quantum Group as K-Theoretic Hall Algebra}

\begin{document}
	
\emergencystretch 3em


\address{Academia Sinica} \email{youhunghsu@gate.sinica.edu.tw}
	
\author[You-Hung Hsu]{You-Hung Hsu}

\maketitle

\begin{abstract}
In this note, we show that the positive part of Arkhipov–Mazin’s 0-affine quantum
group can be realized as the K-theoretic Hall algebra of the type $A$ Dynkin quiver. We then construct a categorical action of this positive part and demonstrate that such an action induces semiorthogonal decompositions on the corresponding weight categories. As a main example, we study the bounded derived category of coherent sheaves on $n$-step partial flag varieties.
\end{abstract}

\section{Introduction}

\subsection{Hall algebras and quantum groups}

Let $\Aa$ be a small abelian category satisfying certain finiteness conditions. The Hall algebra of $\Aa$, denoted by $\Hhh_{\Aa}$, is defined as the $\CC$-vector space with basis given by isomorphism classes of objects in $\Aa$, with multiplication determined by extension data.

A fundamental example is the category of quiver representations over a finite field, denoted by $\Rep_{\FF_{q}}Q$. When $Q$ is a Dynkin quiver (i.e., with underlying graph of type $ADE$ and any orientation), Ringel \cite{Rin} proved that there is an algebra isomorphism 
\begin{equation} \label{finiterealization}
{\Hhh}_{\Rep_{\FF_{q}}Q} \cong \brmU^+_{q}(\cg_{Q})\footnote{For a precise statement, see \cite[Theorem 3.16]{S}}
\end{equation} between the Hall algebra and the positive part of the quantum group, where $\cg_{Q}$ is the simple Lie algebra attached to the quiver $Q$. This realization plays a crucial role in the theory of canonical bases for quantum groups, developed by Kashiwara \cite{K} and Lusztig \cite{L1}, \cite{L2}. For an extensive survey of Hall algebras, see Schiffmann's lecture notes \cite{S}.

\subsubsection{The cohomological/K-theoretic version}

Rather than working with the abelian category of quiver representations directly, one may study the moduli stack of representations of a quiver $Q$. Kontsevich-Soibelman \cite{KS} defined an associative algebra structure on the (Betti) cohomology of this stack; the resulting algebra is called the \textit{cohomological Hall algebra} of $Q$, denoted by $\CoHA_{Q}$. This construction was motivated by the study of BPS state algebras in string theory and their relation to Donaldson–Thomas invariants. Furthermore, they also propose a categorification of $\CoHA$ via using the category of matrix factorizations.

Building on this perspective, Padurariu \cite{T1} introduced the notion of \textit{categorical Hall algebras}, denoted by $\HA_{Q}$, defined as a category of singularities equipped with a monoidal structure. By taking the associated Grothendieck group, he subsequently defined the \textit{K-theoretic Hall algebras}, denoted by $\KHA_{Q}$. In later work, he established several key properties of both $\HA_{Q}$ and $\KHA_{Q}$, including wall-crossing phenomena, a Hopf algebra structure, semiorthogonal decompositions, and a PBW theorem; see \cite{T2, T3} for details.



\begin{remark}
In fact, the constructions in both \cite{KS} and \cite{T1} are more general, where they consider quivers with potential $(Q,W)$.
\end{remark}

\subsubsection{The preprojective algebras and their realization}

Important generalizations of the quantum groups $\brmU_{q}(\cg)$ of simple Lie algebras $\cg$ include Yangians $Y_{\hbar}(\cg)$ and quantum affine algebras $\brmU_{q}(\widehat{\cg})$. Similar to the realization (\ref{finiterealization}), it is natural to ask whether one can realize the positive part of these generalizations as certain Hall algebras.

Nakajima \cite{Nak} constructed actions of quantum affine algebras on the equivariant $K$-theory of quiver varieties, while Varagnolo \cite{V} constructed actions of Yangians on their equivariant homology. These results suggest that the representation theory of quantum affine algebras and Yangians should be accessible through Hall-type constructions. In particular, it is natural to expect that quantum affine algebras are realized via $\KHA$, while Yangians are realized via $\CoHA$. Moreover, there are generalizations of Yangians and quantum affine algebras associated to an arbitrary quiver $Q$. Maulik-Okounkov \cite{MO} defined a Yangian $Y_{\hbar}(\cg_{Q})^{\text{MO}}$, and Okounkov-Smirnov \cite{OS} constructed a quantum affine algebra $\brmU_q(\widehat{\cg_Q})^{\text{OS}}$; both play important roles in the theory of enumerative geometry.

To formulate Hall algebra realizations of the above generalizations of quantum groups, one replaces the moduli stack of representations of a quiver $Q$ by its cotangent stack. It is well-known that this cotangent stack is isomorphic to the moduli stack of representations of the preprojective algebra $\Pi_{Q}$. 

Incorporating the natural $\CC^*$-action, one obtains an associative algebra structure on the $\CC^*$-equivariant cohomology of the moduli stack of $\Pi_{Q}$-representations. The resulting algebra, called the \textit{preprojective cohomological Hall algebra} (denoted by $\CoHA^{\CC^*}_{\Pi_{Q}}$), has long been conjectured to provide an algebraic realizations of the positive part of $Y_{\hbar}(\cg_{Q})^{\text{MO}}$. There has been substantial progress toward this conjecture; see Davison \cite{D1,D2}, Schiffmann–Vasserot \cite{SV1}, and Yang–Zhao \cite{YZ}. Moreover, the conjecture was proved recently, which is the following.

\begin{theorem}[Bota–Davison \cite{BD} and Schiffmann–Vasserot \cite{SV2}]
For a quiver $Q$ of symmetric Kac–Moody type, there are isomorphisms of algebras
\begin{equation*}
\CoHA^{\CC^*}_{\Pi_{Q}} \cong Y^+_{\hbar}(\cg_Q)^{\mathrm{MO}}.
\end{equation*}
\end{theorem}

Analogously, one can equip the $\CC^*$-equivariant K-theory of the moduli stack of $\Pi_{Q}$-representations with an associative multiplication, forming the \textit{preprojective K-theoretic Hall algebra} (denoted by $\KHA^{\CC^*}_{\Pi_{Q}}$). It has likewise been conjectured that the preprojective $\KHA$ realize the positive part of the Okounkov–Smirnov quantum affine algebra $\brmU_q(\widehat{\cg_Q})^{\text{OS}}$.

\begin{conjecture} (\cite[Conjecture 1.2]{T1})
For a quiver $Q$ of symmetric Kac–Moody type, there are isomorphisms of algebras
\begin{equation*}
\KHA^{\CC^*}_{\Pi_Q} \cong \brmU^+_{q}(\widehat{\cg_{Q}})^{\text{OS}}.
\end{equation*}
\end{conjecture} 

Significant progress toward this conjecture has been made; see Pădurariu \cite{T1,T2} and Varagnolo–Vasserot \cite{VV} for details.

\subsection{Main results}

Arkhipov-Mazin \cite{AM} defined an algebra called the \textit{0-affine quantum group}, denoted by $\fU_n$, which can be thought of as a ``$q=0$" degeneration of the quantum affine algebra $\brmU_{q}(\widehat{\GL_n})$. To state their main result, let $\Fl_{n,N}$ be the variety of all $n$-step partial flags in $\CC^N$. There is a natural $\GLL_N(\CC)$ action on it.

They showed that the $\GLL_{N}(\CC)$-equivariant $K$-theory of $\Fl_{n,N} \times \Fl_{n,N}$ carries a convolution algebra structure, which they call the \textit{affine 0-Schur algebra}, denoted by $\brmS^{\aff}_0(n,N)$. The main theorem of \cite{AM} asserts that there exists a surjective algebra homomorphism
\begin{equation*}
    \fU_n \twoheadrightarrow \brmS^{\aff}_0(n,N)\coloneqq K^{\GLL_{N}(\CC)}(\Fl_{n,N} \times \Fl_{n,N})
\end{equation*} for each $N \geq 1$. As a consequence, the 0-affine quantum group $\fU_n$ acts on the equivariant $K$-theory of partial flag varieties.

A natural question, in analogy with Ringel’s theorem and the preprojective Hall algebra constructions discussed above, is whether the positive part of Arkhipov–Mazin’s 0-affine quantum group, denoted by $\fU^+_n$, can be realized in terms of a K-theoretic Hall algebra. 

In this note, we show that $\fU^+_n$ can be realized as the K-theoretic Hall algebra of type $A$ Dynkin quiver. More precisely, we have the following result.


\begin{thmx} [Theorem \ref{main result 1}] \label{Thm1}
Let $\KHA_{A_{n-1}}$ denote the K-theoretic Hall algebra of type $A_{n-1}$ Dynkin quiver. Then, there is an isomorphism of algebras $\phi_n:\fU_{n}^+ \xrightarrow{\simeq} \KHA_{A_{n-1}}$.   
\end{thmx}

The proof will be divided into two parts. 

In the first part, we prove the surjectivity of $\phi_n$ (Proposition \ref{phisurjective}). It relies on the use of the shuffle formula (Proposition \ref{Shuffle}), and some explicit calculations of the shuffle product on degree one elements in the $n=2$ case (Lemma \ref{lemmacalculation0}, \ref{lemmacalculation0'}).  

In the second part we prove the injectivity (Proposition \ref{injectivityphi}). Since that the two algebras $\fU^+_n$ and $\KHA_{A_{n-1}}$ are bi-graded, we will prove it by comparing their bi-graded dimensions. However, a simple check on small examples shows that their bi-graded dimensions are infinite-dimensional. Thus, we have to restrict to their positive/negative parts (Definition \ref{negativesectorKHA}, \ref{positivesectorU+n}).

On the other hand, in \cite{Hsu1} the author defined the \textit{shifted 0-affine algebra}, denoted by $\dot{\Uu}$. Its generators and relations are similar to those of $\fU_n$, but there are two key differences. First, as the notation indicates, $\dot{\Uu}$ is an idempotent modification. Second, unlike $\fU_n$, its presentation involves only finitely many generators and relations.

Despite these differences, since we introduced the notion of a categorical $\dot{\Uu}$-action in \cite{Hsu1}, by combining results from \cite{Hsu2} and \cite{Hsu3} we are able to construct a categorical $\fU_{n}^+$-action and show that it induces semiorthogonal decompositions on weight categories. This leads to our second main result. Before we state it, we include a brief remark on terminology.

\begin{remark}
Although there is an isomorphism of algebras $\fU^+_n \simeq \KHA_{A_{n-1}}$, we will avoid using the term ``categorical $\KHA_{A_{n-1}}$-action". In the literature \cite{T1}, a categorical Hall algebra action usually refers to a functorial action of the categorical Hall algebra itself on a given graded category. Since in our setting the algebra $\fU^+_n$ plays the primary role, we will instead use the term categorical $\fU^+_n$-action throughout this article to avoid confusion.
\end{remark}

\begin{thmx} [Definition \ref{catactkha_n} and Theorem \ref{mainresult_n}]
We define the notion of a categorical $\fU_{n}^+$-action. Moreover, given a categorical $\fU_{n}^+$-action $\Kk$, then there is a semiorthogonal decomposition on each weight category $\Kk(\kk)$ given by the essential images of certain functors.
\end{thmx}

\begin{remark}
The above result can be viewed as an refinement of the main result in \cite[Theorem 4.1]{Hsu2}, where semiorthogonal decompositions were obtained under the stronger assumption of a full categorical $\dot{\Uu}$-action.
\end{remark}

Since there is an action of $\fU_n$ on the $\GLL_N(\CC)$-equivariant K-theory of $n$-step partial flag varieties, it follows in particular that the positive part $\fU^+_n$ also acts on this equivariant K-theory. It is natural to expect that this action on K-theory can be lifted to a categorical $\fU^+_n$-action on the derived category.

In \cite{Hsu1} we established the existence of a categorical $\dot{\Uu}$-action on the bounded derived categories of coherent sheaves on $n$-step partial flag varieties. In particular, this construction induces an action of $\dot{\Uu}$ on their Grothendieck groups.

Moreover, the definitions of the actions of the generators ($e_{i,r}$ and $f_{i,s}$ in \cite{Hsu1}, corresponding respectively to $F_i(r)$ and $E_i(s)$ in \cite{AM}) on equivariant and non-equivariant K-theory coincide. By combining the results of \cite{Hsu1} and \cite{Hsu3}, we arrive at our third main result.

\begin{thmx} [Theorem \ref{main result 3}]
There is a categorical $\fU_{n}^+$-action on the bounded derived categories of coherent sheaves on $n$-step partial flag varieties.
\end{thmx}


\subsection{Further remarks}
Finally, we conclude the introduction with a few remarks and two natural directions for future research that arise from the present work.

The cohomological and K-theoretic Hall algebras of Dynkin quivers have been studied in several earlier works. For example, Kontsevich–Soibelman \cite[Subsection 2.8]{KS} considered $\CoHA_{A_2}$, Rimányi \cite{Rim} treated $\CoHA_{Q}$ for Dynkin quivers $Q$ of type other than $E_8$, and Pădurariu \cite[Subsection 5.2]{T1} investigated $\KHA_{A_{n-1}}$. Our results therefore fit into a broader line of research connecting Hall algebras and representation theory of Dynkin quivers.

In both \cite{AM} and \cite{Hsu1}, the coproduct or Hopf algebra structures on $\fU_n$ and $\dot{\Uu}$ were not investigated. It would be very interesting to introduce such structures and to apply the Drinfeld double construction in order to recover the full algebra. For related developments, see \cite{T2}.

Finally, although in this article we restrict ourselves to type $A$, the same argument is expected to extend to any Dynkin quiver once appropriate analogues of the 0-affine algebra are defined. At present, such algebras have been introduced only for type $A$ (see \cite{AM} and \cite{Hsu1}). Establishing their counterparts for other Dynkin types would provide categorical actions in this broader setting and could lead to semiorthogonal decompositions of derived categories of moduli spaces of (framed) quiver representations.

\subsection{Acknowledgements}
I would like to thank Alexandre Minets for suggesting the statement of Theorem \ref{main result 1} and for proposing the use of the shuffle presentation of 
$\KHA$ in its proof. This idea originated during discussions at Academia Sinica in 2023, whose excellent working environment is gratefully acknowledged. I also thank Tudor Pădurariu and Yu Zhao for helpful comments on an earlier draft of this article. Finally, the author is supported by NSTC grant 113-2115-M-001-003-MY3.

\section{Background materials}
In this section, we recall the basic results about quiver representations and the construction of K-theoretic Hall algebra associated to a quiver (without potentials).

\subsection{Quiver and its representations}

By a \textit{quiver}, we mean a quadruple $Q=(I,E,s,t)$, where $I$ is the vertex set, $E$ is the edge set, and $s,t:E \rightarrow I$ are the source and target maps. Moreover, we let $a_{ij}$ be the number of arrows from vertex $i$ to vertex $j$ for any $i,j \in I$.

We call any $I$-tuple non-negative integers $\alpha=(\alpha^i)_{i \in I} \in \ZZ^I_{\geq 0}$ a \textit{dimension vector}. A \textit{representation} of $Q$ of dimension vector $\alpha$ consists of vector spaces $V^i$ of dimension $\alpha^i$ for each $i \in I$ and $a_{ij}$ linear maps from $V^i$ to $V^j$ for all $i,j \in I$. Then, we denote the space of all representations of $Q$ of dimension $\alpha$ is by $R_\alpha Q$, and it has the following explicit description
\begin{equation*}
    R_{\alpha}Q=\prod_{e \in E} \Hom(V^{s(e)},V^{t(e)}).
\end{equation*} Moreover, we consider the reductive group $G_\alpha Q \coloneqq \prod_{i \in I} \GLL(V^i)$ which acts on $R_\alpha Q$ via conjugation. We identify two representations of in $R_\alpha Q$ if there is an element of $G_{\alpha}Q$ sending one representation to the other. Finally, we denote the stack of representations of $Q$ of dimension vector $\alpha$ to be $\Rep_{\alpha}Q$, which is given by 
\begin{equation} \label{StackRep}
    \Rep_{\alpha}Q \coloneqq [R_{\alpha}Q/G_{\alpha}Q].
\end{equation} For simplicity, we will denote elements in $\Rep_{\alpha}Q$ as $V^{\bullet}=(V^i)_{i \in I}$.

\subsection{K-theoretic Hall algebras for quivers without potentials}
Let $Q=(I,E,s,t)$ be a quiver. For any two dimension vectors $\alpha=(\alpha^i)_{i \in I},\beta=(\beta^i)_{i \in I} \in \ZZ^I_{\geq 0}$, we define $R_{\alpha,\beta}Q$ to be the closed affine subspace of $R_{\alpha+\beta}Q$ which consists of representations of $Q$ of dimension $\alpha+\beta$ that contain a subrepresentation of dimension $\alpha$ ,i.e.
\begin{equation*}
R_{\alpha,\beta}Q \coloneqq \{ (\varphi_{e})_{e \in E} \in R_{\alpha+\beta}Q \ | \ \exists \ W^i \subset V^i, \ \dim W^i=\alpha_i \ \forall \ i, \ \varphi_{e}(W^{s(e)}) \subset W^{t(e)} \ \forall \ e\in E \}.
\end{equation*} Similarly, we consider the parabolic subgroup $G_{\alpha,\beta}Q \coloneqq \prod_{i} P_{\alpha^i,\beta^i} \subset G_{\alpha+\beta}Q=\prod_{i} \GLL(V^{i})$, where $P_{\alpha^i,\beta^i} \subset \GLL(V^i)$ is the subgroup that stabilizes the subspace $W^i \subset V^i$. Clearly, $G_{\alpha,\beta}Q$ acts on $R_{\alpha,\beta}Q$ by conjugation. Then, we define the stack 
\begin{equation*}
    \Fl_{\alpha,\beta}Q \coloneqq [R_{\alpha,\beta}Q/G_{\alpha,\beta}Q],
\end{equation*} and will simply denote the elements in $\Fl_{\alpha,\beta}Q$ to be $W^{\bullet} \subset V^{\bullet} = (W^i \subset V^i)_{i \in I}$.

There is a natural correspondence
\begin{equation} \label{corr}
    \xymatrix{
    & \Fl_{\alpha,\beta}Q \ar[ld]_{q} \ar[rd]^{p} \\
    \Rep_{\alpha}Q \times \Rep_{\beta}Q &  & \Rep_{\alpha+\beta}Q
    }
\end{equation} where the maps $p$, $q$ are given by 
\begin{equation*}
    q(W^{\bullet} \subset V^{\bullet})=(W^{\bullet}, V^{\bullet}/W^{\bullet}), \ p(W^{\bullet} \subset V^{\bullet})=V^{\bullet}.
\end{equation*}

For any space $X$, we denote $K(X)$ to be the complexified Grothendieck group of $X$, respectively. The maps $p,q$ induce pushforwards and pullbacks in K-theory. Then, from the correspondence (\ref{corr}), we obtain the following maps
\begin{align*}
    m^K_{\alpha,\beta}&:K(\Rep_{\alpha}Q) \otimes K(\Rep_{\beta}Q) \rightarrow K(\Rep_{\alpha+\beta}Q). 
\end{align*}

\begin{definition} 
For a quiver $Q$, its K-theoretic Hall algebra is defined to be the vector space $\bigoplus_{\alpha}K(\Rep_{\alpha}Q)$ equipped with the product $m^K_{\alpha,\beta}$.
\end{definition}

\begin{proposition} \cite[Theorem 3.3]{T1}
$\KHA_{Q}\coloneqq (\bigoplus_{\alpha}K(\Rep_{\alpha}Q),m^{K}_{\alpha,\beta})$ is an associative algebra.
\end{proposition}

\begin{remark}
If we replace K-theory by cohomology, then we obtain the vector space $\bigoplus_{\alpha}H^*(\Rep_{\alpha}Q)$. Similarly, the correspondence (\ref{corr}) induces a product denoted by $m^H_{\alpha,\beta}$. By \cite[Theorem 1]{KS}, $\CoHA_Q \coloneqq (\bigoplus_{\alpha}H^*(\Rep_{\alpha}Q),m^{H}_{\alpha,\beta})$ is an associative algebra, which is called the cohomological Hall algebra of $Q$. 
\end{remark}

\subsection{The shuffle formula for the product}

It is well-known that for the general linear group $\GLL_n(\CC)$, the equivariant K-theory of a point is given by 
\begin{equation*}
    K^{\GLL_n(\CC)}(pt) \cong \CC[x^{\pm 1}_1,...,x^{\pm 1}_{n}]^{\fS_n}
\end{equation*} where $\fS_n$ is the symmetric group of $n$ letters, which is isomorphic to the Weyl group of $\GLL_n(\CC)$.

From the description of the representation stack in (\ref{StackRep}), as a vector space, we have 
\begin{equation} \label{polyiso}
    K(\Rep_{\alpha}Q) \cong \bigotimes_{i \in I} K^{\GLL_{\alpha^i}(\CC)}(pt) \cong \bigotimes_{i \in I} \CC[x^{\pm 1}_{i1},...,x^{\pm 1}_{i\alpha^i}]^{\fS_{\alpha^i}}.
\end{equation} For simplicity, we will just write $K(\Rep_{\alpha}Q) \cong \CC[x^{\pm 1}_{ij}]^{\fS_{\alpha}}$ where $i \in I$, $1 \leq j \leq \alpha^i$, and $\fS_{\alpha}=\prod_{i \in I} \fS_{\alpha^i}$.

Under the isomorphism (\ref{polyiso}), there is an explicit formula (called the shuffle formula) for the multiplication $m^{K}_{\alpha,\beta}$ in the K-theoretic Hall algebra, which is due to \cite{T1}. 

\begin{proposition}\cite[Proposition 3.4]{T1} \label{Shuffle}
For any two dimension vectors $\alpha, \ \beta \in \NN^I$, given $f \in K(\Rep_{\alpha}Q)$ and $g \in K(\Rep_{\beta}Q)$. Their multiplication in $\KHA_{Q}$ is $m^K_{\alpha,\beta}(f,g)$, which is an element in $K(\Rep_{\alpha+\beta}Q)$ with explicit formula given by 
\begin{equation*}
m^K_{\alpha,\beta}(f,g)(x_{ij})= \frac{1}{|\fS_{\alpha}| |\fS_{\beta}|} \Sym \Biggl ( f(x_{ij})g(x_{i'j'}) \frac{\prod\limits_{\substack{i \in I \\ 1 \leq j \leq \alpha^i}} \prod\limits_{\substack{i' \in I \\ \alpha^{i'}+1 \leq j' \leq \alpha^{i'}+\beta^{i'}}} (1-x_{ij}x^{-1}_{i'j'})^{c(i,i')}}{\prod\limits_{i \in I} \prod\limits_{\substack{1 \leq j \leq \alpha^i \\ \alpha^i+1 \leq j' \leq \alpha^i+\beta^i}} (1-x_{ij}x^{-1}_{ij'})}   \Biggr)
\end{equation*} where $c(i,i')$ is the number of arrow from $i$ to $i'$ in $Q$.
\end{proposition}

\begin{remark}
For more studies on shuffle algebras, we refer readers to \cite{N1} and \cite{N2}.
\end{remark}

\begin{remark}
Similarly, there is a shuffle formula for the product of $\CoHA_Q$, see \cite[Theorem 2]{KS}.
\end{remark}


\begin{example} \cite[Corollary 3.5]{T1}
For $Q=\xymatrix{\cdot \ar@(dr,ur)}$ \ \ \ which is the Jordan quiver, and consider the action of $\CC^*$ which scales representations of $Q$ with weight one. Then, there is an isomorphism
\begin{equation}
\KHA^{\CC^*}_{Q} \cong \brmU^+_{q}(L\SL_2)   
\end{equation}
\end{example}

\section{0-affine quantum groups and their action}

In this section, we first recall the definition of 0-affine quantum groups defined in \cite{AM}. Then, we mention the main result in \textit{loc. cit.} where there is a surjective algebra homomorphism from the 0-affine quantum group to the affine 0-Schur algebra.

\subsection{Definition of the algebra}


\begin{definition} \cite[Definition 1.1]{AM} \label{0affinequantumgroup}
The \textit{0-affine quantum group}, denoted by $\fU_n=\brmU_0(\GL_n[t,t^{-1}])$, is generated by the following elements
\begin{equation*}
    \{e_{i,r}, f_{i,r}, h_{n,r}\}_{\ 1 \leq i\leq n-1}^{r\in \ZZ}.
\end{equation*} subject to the following relations
\begin{enumerate}
    \item Relations on $e$ generators
\begin{align} \label{er}
 \begin{split} 
e_{i,r}e_{i,s}&=-e_{i,s-1}e_{i,r+1}, \ \text{for all} \ r,s \in \ZZ,\ \text{and} \ 1 \leq i \leq n-1. \\
e_{i+1,s}e_{i,r}&=e_{i,r}e_{i+1,s}-e_{i,r+1}e_{i+1,s-1}, \ \text{for all} \ r,s \in \ZZ, \ \text{and} \ 1 \leq i \leq n-2. \\
e_{i,r}e_{j,s}&=e_{j,s}e_{i,r}, \ \text{for all} \ r,s \in \ZZ,\ 1 \leq i,j \leq n-1 \ \text{and} \ |i-j| \geq 2.
\end{split}
\end{align}

\item Relations on $f$ generators
\begin{align} \label{fs}
\begin{split}
f_{i,r}f_{i,s}&=-f_{i,s+1}f_{i,r-1}, \ \text{for all} \ r,s \in \ZZ,\ \text{and} \ 1 \leq i \leq n-1. \\
f_{i,r}f_{i+1,s}&=f_{i+1,s}f_{i,r}-f_{i+1,s-1}f_{i,r+1}, \ \text{for all} \ r,s \in \ZZ,\ \text{and} \ 1 \leq i \leq n-2. \\
f_{i,r}f_{j,s}&=f_{j,s}f_{i,r}, \ \text{for all} \ r,s \in \ZZ,\ 1 \leq i,j \leq n-1 \ \text{and} \ |i-j| \geq 2.
\end{split}
\end{align}

\item Relations between $e$ and $f$ generators
\begin{align} \label{erfs}
\begin{split}
[e_{i,r},f_{j,s}]&=0, \ \text{for all} \ r,s \in \ZZ, \ \text{and} \ 1 \leq i,j \leq n-1 \ \text{with} \ i \neq j, \\
e_{i,r}f_{i,s}-f_{i,s}e_{i,r}&=e_{i,r'}f_{i,s'}-f_{i,s'}e_{i,r'}, \ \text{whenever} \ r+s=r'+s',\ \text{and} \ 1 \leq i\leq n-1.
\end{split}
\end{align} 
We define $h_{i,r+s} \coloneqq [e_{i,r},f_{i,s}]$ for all $r,s \in \ZZ$ and $1 \leq i \leq n-1$.

\item Relations on $h$ generators
\begin{equation}
h_{i,r}h_{j,s}=h_{j,s}h_{i,r}, \ \text{for all} \ 1 \leq i,j \leq n, \ r,s \in \ZZ.
\end{equation}

\item Relations between $h$ and $e$ generators
\begin{align} \label{hres}
\begin{split}
h_{i,r}e_{i,s}&=-e_{i,s-1}h_{i,r+1}, \ \text{for all} \ r,s \in \ZZ,\ \text{and} \ 1 \leq i \leq n-1 \\
h_{n,s}e_{n-1,r}&=e_{n-1,r}h_{n,s}-e_{n-1,r+1}h_{n,s-1}, \ \text{for all} \ r,s \in \ZZ, \\
h_{n,r}e_{i,s}&=e_{i,s}h_{n,r}, \ \text{for all} \ r,s \in \ZZ, \ \text{and} \ 1 \leq i \leq n-2.  
\end{split}
\end{align}

\item Relations between $h$ and $f$ generators
\begin{align} \label{hrfs}
\begin{split}
h_{i,r}f_{i,s}&=-f_{i,s+1}h_{i,r-1}, \ \text{for all} \ r,s \in \ZZ,\ \text{and} \ 1 \leq i \leq n-1 \\
f_{n-1,r}h_{n,s}&=h_{n,s}f_{n-1,r}-h_{n,s-1}f_{n-1,r+1}, \ \text{for all} \ r,s \in \ZZ, \\
h_{n,r}f_{i,s}&=f_{i,s}h_{n,r},  \ \text{for all} \ r,s \in \ZZ, \ \text{and} \ 1 \leq i \leq n-2. 
\end{split}
\end{align}
\end{enumerate}
\end{definition}

\begin{remark}
Note that there is also an algebra called \textit{shifted 0-affine algebra}, denoted by $\dot{\Uu}$, defined in \cite{Hsu1} with similar generators and relations. However, $\dot{\Uu}$ is different from $\fU$ in two ways. First, $\dot{\Uu}$ is an idempotent modification. Second, $\dot{\Uu}$ is finitely generated.
\end{remark}

\begin{definition} \label{positivepart}
We define the positive part of $\fU_n$, denoted by $\fU^+_n$, to be the subalgebra generated by $e_{i,r}$ with $1 \leq i \leq n-1$ and $r \in \ZZ$, subject to the relation (\ref{er}).
\end{definition}

\subsection{Surjections to the affine 0-Schur algebras}

Let $N \geq 2$ be a positive integer. We say $\kk$ is a weak composition of $N$ of lenth $n$ if $\kk=(k_1,...,k_n) \in \ZZ^n_{\geq 0}$ and $\sum_{i}k_i=N$. We write $\kk \vDash N$ and $l(\kk)=n$ in this case. We denote $C(n,N)$ to be the set of all weak compositions of $N$ of length $n$, i.e.
\begin{equation*}
C(n,N)=\{\kk=(k_1,...,k_n) \ | \ \kk \vDash N, \ l(\kk)=n \}.
\end{equation*} 

We denote $\Fl_{n,N}$ to be the variety of all $n$-step partial flags in $\CC^N$, i.e. 
\begin{equation*}
    \Fl_{n,N}=\{ 0\subset V_1 \subset ... \subset V_n=\CC^N\}.
\end{equation*} Then, clearly, there is a decomposition $\Fl_{n,N}=\bigsqcup_{\kk \in C(n,N)} \Fl_{\kk}(\CC^N)$, where $\Fl_{\kk}(\CC^N)$ is the variety of $\kk$-partial flags defined as follows 
\begin{equation} \label{eq fl}
\Fl_{\kk}(\CC^N) \coloneqq \{V_{\bullet}=(0=V_{0} \subset V_1 \subset ... \subset V_{n}=\CC^N) \ | \ \tdim V_{i}/V_{i-1}=k_{i} \ \text{for} \ \text{all} \ i\}.
\end{equation}

\begin{proposition}\cite[Definition 1.6, Theorem 1.7]{AM}
There is a convolution algebra structure on $K^{\GLL_{N}(\CC)}(\Fl_{n,N} \times \Fl_{n,N})$, called the affine 0-Schur algebra, denoted by $\brmS_0^{\aff}(n,N)$. Moreover, for any positive integer $N$, there is a surjective algebra homomorphism
\begin{equation*}
\fU_{n} \twoheadrightarrow \brmS_0^{\aff}(n,N).
\end{equation*}
\end{proposition}

\begin{remark}
If we consider the cotangent bundle $T^*\Fl_{n,N}$ and denote $\Nn$ to be the nilpotent cone, then the fiber product $Z_{n,N} \coloneqq T^*\Fl_{n,N} \times_{\Nn} T^*\Fl_{n,N}$ is the Steinberg variety. Its $\GLL_N(\CC) \times \CC^*$-equivariant K-theory carries a convolution algebra structure, which is called the \textit{affine q-Schur algebra}, denoted by $\brmS^{\aff}_q(n,N)=K^{\GLL_N(\CC) \times \CC^*}(Z_{n,N})$. Then, due to Vasserot \cite{V'}, there is a surjective algebra homomorphism
\begin{equation*}
    \brmU_{q}(\widehat{\GL}_n) \twoheadrightarrow \brmS^{\aff}_q(n,N).
\end{equation*}
\end{remark}

\section{Main results}

In this section, we prove the main results in this article. 

\subsection{An algebra isomorphism}
Let $Q(A_n)$ denote the type $A_n$ Dynkin quiver, which is the following diagram with orientation
\begin{equation}
   \xymatrix@R=0.1pc{
        1 & 2 & ... & n-1 & n \\
        \bullet & \bullet \ar[l] & ... \ar[l] & \bullet \ar[l] & \bullet \ar[l]. 
    }
\end{equation} Since we will only consider the type $A$ Dynkin quivers in the rest of this article, we will simply denote its K-theoretic Hall algebra by $\KHA_{A_{n}}$. Recall that, from (\ref{polyiso}), $\KHA_{A_n}$ is graded by the dimension vector and its graded piece is given by 
\begin{equation*}
K(\Rep_{\alpha}Q(A_n)) \cong \CC[x^{\pm 1}_{ij}]^{\fS_{\alpha}}
\end{equation*} where $1 \leq i \leq n$, $1 \leq j \leq \alpha^i$, and $\fS_{\alpha}=\prod_{1 \leq i \leq n}\fS_{\alpha^i}$. Moreover, instead of using $m^{K}_{\alpha,\beta}$, we will simply denote the multiplication in $\KHA_{A_n}$ by $\ast$. Then, we have the following result.

\begin{theorem} \label{main result 1}
There is an algebra isomorphism $\phi_{n+1}:\fU^+_{n+1} \rightarrow \KHA_{A_n}$ given by $\phi_{n+1}(e_{i,r})=x^{-r}_{i,1}$ for all $1\leq i \leq n$ and $r \in \ZZ$.
\end{theorem}


First, we prove the surjectivity of $\phi_{n+1}$.

\subsubsection{Surjectivity}

We need to know more about the algebra structure of $\KHA_{A_n}$. In the following, we prove a few results that are well-known to experts, but we decide to provide proofs for completeness. 

We start from the simplest case where $n=1$. In this case, the K-theoretic Hall algebra of the type $A_1$ quiver (a point) is 
\begin{equation*}
\KHA_{A_1}=\bigoplus_{r \geq 0} K(\Rep_rQ(A_1)) \cong \bigoplus_{r \geq 0} \CC[x^{\pm 1}_1,...,x_r^{\pm 1}]^{\fS_r}
\end{equation*}


Let $\Lambda_r=\CC[x_1^{\pm1},\dots,x_r^{\pm1}]^{\fS_r}$ and set  $\omega(z)=\frac{1}{1-z}$. Then, $\KHA_{A_1}=\bigoplus_{r \geq 0}\Lambda_r$ is the algebra equipped with the shuffle product
\begin{equation} \label{shuffleproductA1}
(f \ast g)(x_1,\dots,x_{m+n})
\coloneqq \frac{1}{m!n!}\Sym\Big[f(x_1,\dots,x_m)\,g(x_{m+1},\dots,x_{m+n})
\prod_{1\leq i\leq m<j \leq m+n}\omega \Big(\frac{x_i}{x_j}\Big)\Big]
\end{equation} for all $f \in \Lambda_m$, $g \in \Lambda_n$ and $\Sym$ denotes the symmetrization operator, i.e., 
\begin{equation*}
    \Sym(f(x_1,...,x_k))=\sum_{\sigma \in \fS_k}f(x_{\sigma(1)},...,x_{\sigma(k)}).
\end{equation*}


First, we have the following two calculation results.

\begin{lemma}\label{lemmacalculation0}
Consider the constant function $1 \in K(\Rep_1Q(A_1))$. Then, for each $r \geq 1$, the following equation holds in $K(\Rep_rQ(A_1))$
\begin{equation*}
    1 \ast 1 \ast ... \ast 1 =1 
\end{equation*} where the right hand side is $1$ itself shuffle product $r$ times.
\end{lemma}

\begin{proof}
We prove this by induction on $r$. Since the case $r=1$ is obvious, so we check it for $r=2$.
\begin{equation*}
    1 \ast 1 =\frac{1}{1!1!}\Sym\Big(1\omega(\frac{x_1}{x_2})\Big)=\Sym(\frac{x_2}{x_2-x_1})=\frac{x_2}{x_2-x_1}-\frac{x_1}{x_2-x_1}=1.
\end{equation*}

Now, we assume that the identity holds for $r \leq k-1$. Then, when $r=k$, we have 
\begin{align}
\begin{split} \label{induction}
1 \ast 1 \ast ... \ast 1 &= (1 \ast ... \ast 1) \ast 1 =1 \ast 1 \\
&= \frac{1}{(k-1)!1!}\Sym\Big( \prod_{1 \leq i \leq k-1<j \leq k} \omega(\frac{x_i}{x_j}) \Big)=\Sym\Big( \prod_{1 \leq i \leq k-1} \frac{x_k}{x_k-x_i}\Big).
\end{split}
\end{align} where we use the induction hypothesis in the second equality.

We denote 
\begin{equation*}
g(x_1,...,x_k)=\prod_{1 \leq i \leq k-1} \frac{x_k}{x_k-x_i}=\frac{x^{k-1}_k}{(x_k-x_1)...(x_k-x_{k-1})}
\end{equation*} Note that $g(x_{\sigma(1)},...,x_{\sigma(k)})=g(x_1,...,x_k)$ for any $\sigma$ in the Young subgroup $\fS_{k-1} \times \fS_1 \leq \fS_k$. Let $\tau_1=e, \tau_2=(k-1,k),\tau_3=(k-2,k),...,\tau_k=(1,k)$ be the representatives in $\fS_k/(\fS_{k-1} \times \fS_1)$. Then, 
\begin{align*}
&\Sym(g(x_1,...,x_k))=(k-1)!\sum_{i=1}^{k} g(x_{\tau_i(1)},...,x_{\tau_i(k)}) \\
&=(k-1)!\Big[\frac{x^{k-1}_k}{(x_k-x_1)...(x_k-x_{k-1})}+\frac{x^{k-1}_{k-1}}{(x_{k-1}-x_1)...(x_{k-1}-x_k)}+...+\frac{x^{k-1}_1}{(x_1-x_k)...(x_1-x_{k-1})} \Big] \\
&=(k-1)!\frac{1}{\prod_{1 \leq i <j \leq k}(x_j-x_i)} \\
&\Big[x^{k-1}_k\prod_{1 \leq i<j \leq k-1}(x_j-x_i)-x^{k-1}_{k-1}\prod_{1\leq i<j\leq k,\ j \neq k-1 }(x_j-x_i)+....+(-1)^{k-1}x^{k-1}_1 \prod_{2\leq i<j\leq k}(x_j-x_i)\Big]\\
&=(k-1)!.
\end{align*} Thus, we obtain (\ref{induction}) is equal to $1$.
\end{proof}

The following result compute the shuffle products for degree one elements, see \cite[(6.1)]{N2} for a similar statement.

\begin{lemma}\label{lemmacalculation0'}
For all $r \geq 1$ and $a_1,...,a_r \in \ZZ$, we have the following equation
\begin{equation*}
    x^{a_1}_1 \ast ... \ast x^{a_r}_1=\Sym\Big(x^{a_1}_1...x^{a_r}_r \prod_{1 \leq i <j \leq r} \omega(\frac{x_i}{x_j}) \Big)
\end{equation*}
\end{lemma}

\begin{proof}
We prove this by induction on $r$. Clearly, the equation holds for $r=1$, thus we check the case where $r=2$.
\begin{equation*}
    x_1^{a_1} \ast x^{a_2}_1 =\Sym\Big(x^{a_1}_1x_2^{a_2}\omega(\frac{x_1}{x_2}) \Big).
\end{equation*}

Now, we assume that the identity holds for $r \leq k-1$. Then, when $r=k$, we have 
\begin{align*}
x^{a_1}_1 \ast ... \ast x^{a_k}_1&=(x^{a_1}_1 \ast ... \ast x_1^{a_{k-1}}) \ast x^{a_k}_1=\Bigg\{\Sym\Big(x^{a_1}_1...x^{a_{k-1}}_{k-1}\prod_{1 \leq i <j \leq k-1} \omega(\frac{x_i}{x_j}) \Big) \Bigg\} \ast x^{a_k}_1 \\
&=\frac{1}{(k-1)!1!}\Sym\Bigg( \Sym\Big(x^{a_1}_1...x^{a_{k-1}}_{k-1}\prod_{1 \leq i <j \leq k-1} \omega(\frac{x_i}{x_j}) \Big)x_k^{a_k} \prod_{1 \leq i \leq k-1<j \leq k} \omega(\frac{x_i}{x_j}) \Bigg) \\
& =\Sym\Big( x^{a_1}_1...x^{a_{k-1}}_{k-1}x_k^{a_k}\prod_{1 \leq i <j \leq k} \omega(\frac{x_i}{x_j}) \Big).
\end{align*}
\end{proof}

Then, we have the following

\begin{lemma} \label{A1generatedegreeone}
The algebra $\KHA_{A_1}$ is generated by the degree one elements under the shuffle product $\ast$, i.e., for each $r \geq 0$, every element $f \in \Lambda_r$ can be written as a linear combination of $x^{d_1}_1 \ast ... \ast x^{d_r}_1$.
\end{lemma}

\begin{proof}
Given $f \in \Lambda_r=\CC[x_1^{\pm 1},...,x_r^{\pm 1}]^{\fS_r}$, then $f$ can be written as a linear combination of monomials
\begin{equation*}
f=\sum_{d_1,...,d_r}c_{d_1...d_r}x^{d_1}_1...x^{d_r}_r
\end{equation*} for some $d_1,...,d_r \in \ZZ$ and constant $c_{d_1...d_r} \in \CC$.

Then, by Lemma \ref{lemmacalculation0} and Lemma \ref{lemmacalculation0'} with $a_1=...=a_r=0$, we have 
\begin{align*}
f&=f \cdot 1 =f(1 \ast 1 \ast ... \ast 1) =f\Sym\Big(\prod_{1 \leq i<j \leq r} \omega(\frac{x_i}{x_j})\Big) \\
&=\Sym\Big(f\prod_{1 \leq i<j \leq r} \omega(\frac{x_i}{x_j})\Big)=\Sym\Big((\sum_{d_1,...,d_r}c_{d_1...d_r}x^{d_1}_1...x^{d_r}_r)\prod_{1 \leq i<j \leq r} \omega(\frac{x_i}{x_j})\Big) \\
&=\sum_{d_1,...,d_r}c_{d_1...d_r}\Sym\Big((x^{d_1}_1...x^{d_r}_r)\prod_{1 \leq i<j \leq r} \omega(\frac{x_i}{x_j})\Big) \\
&=\sum_{d_1,...,d_r}c_{d_1...d_r}x^{d_1}_1 \ast ... \ast x^{d_r}_1
\end{align*} where we use the property that $f$ is symmetric in the fourth equality and Lemma \ref{lemmacalculation0'} in the last equality.


\end{proof}

\begin{remark}
For a related result, we refer the readers to \cite[Theorem 2.5]{N2}.
\end{remark}


Now, we move to the general $n$ case. To describe the shuffle product that generalize (\ref{shuffleproductA1}), we need to introduce a kernel that serves as the twist and generalizes $\omega(z)=\frac{1}{1-z}$. Let 
\begin{equation*}
\omega_{ij}(z)=\begin{cases}
\frac{1}{1-z} & \text{if} \ i=j, \\
1-z & \text{if} \ |i-j|=1, \\
1 & \text{if} \ |i-j| \geq 2.
\end{cases}
\end{equation*} 


Then, by Proposition \ref{Shuffle}, the shuffle product in $\KHA_{A_{n}}$ can be written as
\begin{align}\label{shuffleproductAn}
\begin{split}
&(f \ast g)(x_{ij}) \\
& = \frac{1}{|\fS_\alpha||\fS_{\beta}|} \Sym\Big[f(y_{ij})g(z_{ij})\prod_{i=1}^n \prod\limits_{\substack{1 \leq j \leq \alpha^i \\ \alpha^i+1 \leq j' \leq \alpha^i+\beta^i}} \omega_{ii}\Big(\frac{x_{ij}}{x_{ij'}} \Big) \prod_{|i-i'|=1}\prod\limits_{\substack{1 \leq j \leq \alpha^i \\ \alpha^{i'}+1 \leq j' \leq \alpha^{i'}+\beta^{i'}}} \omega_{ii'}\Big(\frac{x_{ij}}{x_{i'j'}}\Big) \Big] 
\end{split}
\end{align} for all $f \in K(\Rep_{\alpha}Q(A_n))=\CC[y^{\pm 1}_{ij}]^{\fS_{\alpha}}$ and $g \in K(\Rep_{\beta}Q(A_n))=\CC[z^{\pm 1}_{ij}]^{\fS_{\beta}}$, where $x_{ij}=y_{ij}$ for $1 \leq j \leq \alpha^i$ and $x_{i\alpha^i+j'}=z_{ij'}$ for $1 \leq j' \leq \beta^i$.

If we denote $\omega_i$ to be the dimension vector with $1$ in vertex $i$ and zero elsewhere, then considering the following 
\begin{equation*}
   R_i \coloneqq  \bigoplus_{r \geq 0} K(\Rep_{r\omega_i}Q(A_n)), \ 1 \leq i \leq n
\end{equation*} they are subalgebras of $\KHA_{A_n}$ that all isomorphic to $\KHA_{A_1}$. 


Then, we have the following result.

\begin{proposition} \label{PBWKHAAn}
The algebra $\KHA_{A_n}$ is generated by the $\omega_i$-graded piece, i.e., the $\ast$ multiplication (from left to right) induces isomorphism
\begin{equation*}
\mu:R_1 \otimes ... \otimes R_n \xrightarrow{\ast} \KHA_{A_n} 
\end{equation*}
\end{proposition}

\begin{proof}
Given $f \in \KHA_{A_n}$, then $f$ can be written as a finite sum $f=\sum_{i=1}^k f_i$ where $f_i \in K(\Rep_{\alpha_i}Q(A_n)) \cong \CC[x^{\pm 1}_{ij}]^{\fS_{\alpha_i}}$ for all $1 \leq i \leq k$.

Let $\alpha_i=(\alpha_i^1,...,\alpha_i^n)$ for all $1 \leq i \leq k$. Then, we have $f_i=\sum_{r=1}^{l_i}\prod_{j=1}^ng^r_{ij}$ for some $l_i\geq 0$, where $g^r_{ij} \in \CC[x^{\pm 1}_{j1},...,x^{\pm 1}_{j\alpha_i^j}]^{\fS_{\alpha^j_i}}$. Then, we pick $\sum_{r=1}^{l_i}g^r_{i1} \otimes ... \otimes g^r_{in} \in R_1 \otimes ... \otimes R_n$. Since each $g^r_{ij}$ is already symmetric with respect to $\fS_{\alpha_i^j}$ and the conditions in the shuffle product in (\ref{shuffleproductAn}) is empty, we obtain 
\begin{equation*}
\sum_{r=1}^{l_i}g^r_{i1} \otimes ... \otimes g^r_{in} \mapsto \sum_{r=1}^{l_i}g^r_{i1} \ast  ... \ast g^r_{in} =\sum_{r=1}^{l_i}\prod_{j=1}^ng^r_{ij}=f_i.
\end{equation*} Extend this linearly, we obtain that it is surjective. Moreover, since the map sends tensor product of polynomials into usual product, it is injective. Thus it is an isomorphism.

\end{proof}

\begin{remark}
We refer the readers to \cite[Corollary 5.4]{T2} for a similar result, and \cite[Theorem 11.2]{Rim} for a result in the $\CoHA$ version.
\end{remark}

The following proves part of Theorem \ref{main result 1}.

\begin{proposition} \label{phisurjective}
The map $\phi_{n+1}$ in Theorem \ref{main result 1} is a surjective algebra homomorphism. 
\end{proposition}

\begin{proof}


Note that, by Lemma \ref{A1generatedegreeone}, we know that $\KHA_{A_1}$ is generated by degree-one elements, i.e. $x^r \in K(\Rep_{1}Q(A_1))=\CC[x^{\pm 1}]$. Moreover, by Proposition \ref{PBWKHAAn}, $\KHA_{A_n}$ is generated by the $\omega_i$-graded piece subalgebra that isomorphic to $\KHA_{A_1}$. Thus, $\KHA_{A_n}$ is generated by those degree one elements $x^r_{i,1} \in K(\Rep_{\omega_i}Q(A_n))$, and $\phi_{n+1}$ maps the generators of $\fU_{n+1}^+$ to the generators of $\KHA_{A_n}$. 

It remains to check that $\phi$ is an algebra homomorphism, i.e., the relations in $\fU^+_{n+1}$ hold in $\KHA_{A_n}$.

Fix a vertex $1 \leq i \leq n$, consider $x_{i,1}^{-r}, \ x_{i,1}^{-s} \in K(\Rep_{\omega_i}Q(A_n))$, we have 
\begin{equation*}
 x_{i,1}^{-r} \ast x_{i,1}^{-s} =\frac{1}{1!1!} \Sym \biggl(x_{i,1}^{-r}x_{i,2}^{-s} \frac{1}{1-x_{i,1}x^{-1}_{i,2}} \biggr)=\frac{x_{i,1}^{-r}x_{i,2}^{-s}}{1-x_{i,1}x^{-1}_{i,2}}+\frac{x_{i,2}^{-r}x_{i,1}^{-s}}{1-x_{i,2}x^{-1}_{i,1}} =\frac{x_{i,1}^{-r}x_{i,2}^{-s+1}-x_{i,2}^{-r}x_{i,1}^{-s+1}}{x_{i,2}-x_{i,1}}.
\end{equation*} Similarly, 
\begin{equation*}
 x_{i,1}^{-s+1} \ast x_{i,1}^{-r-1} =\frac{1}{1!1!} \Sym \biggl(x_{i,1}^{-s+1}x_{i,2}^{-r-1} \frac{1}{1-x_{i,1}x^{-1}_{i,2}} \biggr)=\frac{x_{i,1}^{-s+1}x_{i,2}^{-r-1}}{1-x_{i,1}x^{-1}_{i,2}}+\frac{x_{i,2}^{-s+1}x_{i,1}^{-r-1}}{1-x_{i,2}x^{-1}_{i,1}} =\frac{x_{i,1}^{-s+1}x_{i,2}^{-r}-x_{i,2}^{-s+1}x_{i,1}^{-r}}{x_{i,2}-x_{i,1}}.
\end{equation*} Thus, we obtain $x_{i,1}^{-r} \ast x_{i,1}^{-s}=-x_{i,1}^{-s+1} \ast x_{i,1}^{-r-1}$, which corresponds to $e_{i,r}e_{i,s}=-e_{i,s-1}e_{i,r+1}$ (the first relation in (\ref{er})).

Next, fix two vertices $i$ and $i+1$ with an arrow from $i+1$ to $i$. Consider $x_{i,1}^{-r} \in  K(\Rep_{\omega_i}Q(A_n))$ and $x_{i+1,1}^{-s} \in  K(\Rep_{\omega_{i+1}}Q(A_n))$. Then, 
\begin{align*}
   x_{i+1,1}^{-s} \ast x_{i,1}^{-r} = \frac{1}{1!1!}\Sym \biggl( x_{i+1,1}^{-s}x_{i,1}^{-r}\frac{1-x_{i+1,1}x_{i,1}^{-1}}{1}\biggr)=x_{i+1,1}^{-s}x_{i,1}^{-r}-x_{i+1,1}^{-s+1}x_{i,1}^{-r-1}.
\end{align*} On the other hand, it is standard to calculate
\begin{equation*}
 x_{i,1}^{-r} \ast  x_{i+1,1}^{-s}=   x_{i,1}^{-r}x_{i+1,1}^{-s}, \ \text{and} \ x_{i,1}^{-r-1} \ast  x_{i+1,1}^{-s+1}=  x_{i,1}^{-r-1}x_{i+1,1}^{-s+1}.
\end{equation*} Thus, we obtain $ x_{i+1,1}^{-s} \ast x_{i,1}^{-r}=x_{i,1}^{-r} \ast  x_{i+1,1}^{-s}-x_{i,1}^{-r-1} \ast  x_{i+1,1}^{-s+1}$, which corresponds to the second relation $e_{i+1,s}e_{i,r}=e_{i,r}e_{i+1,s}-e_{i,r+1}e_{i+1,s-1}$ in (\ref{er}).

Finally, for two vertexes $i$ and $j$ such that there are no arrows between them, it is easy to check that $x_{i,1}^{-r} \ast  x_{j,1}^{-s}=x_{j,1}^{-s} \ast  x_{i,1}^{-r}$, and it corresponds to the third relation  in (\ref{er}). 

Hence, $\phi_{n+1}$ is a surjective algebra homomorphism.
\end{proof}

\subsubsection{Comparison of graded dimensions}
Since we already show that $\phi_{n+1}$ is surjective, to show that $\phi_{n+1}$ is an isomorphism, we need to compare the bi-graded dimensions on both sides of the map $\phi_{n+1}$.

Let us look at the $\KHA$ side first. Note that, by definition, $\KHA_{A_n}$ is graded by the dimension vector $\alpha=(\alpha^1,...,\alpha^n) \in \ZZ^n_{\geq 0}$. Moreover, from the description of the graded piece in (\ref{polyiso}), we know that, beside the grading by dimension vectors, we also have the internal/loop grading from the degree of the polynomials.

Then, for each $m \in \ZZ$, we define the following bi-graded piece of $\KHA_{A_n}$
\begin{equation*}
K(\Rep_{\alpha}Q(A_{n}))_m=\Span\bigg\{ \Sym\bigg(\prod_{1 \leq i \leq n}\prod_{1 \leq a \leq \alpha^i} x^{r_{i,a}}_{i,a}\bigg) \in \CC[x^{\pm 1}_{i,j}]^{\fS_{\alpha}} \ \bigg| \ \sum_{i=1}^n\sum_{a=1}^{\alpha^i} r_{i,a}=m \bigg\} 
\end{equation*} which is the subspace spanned by symmetric Laurent monomials of degree $m$.

On the other hand, for the 0-affine quantum group $\fU^+_{n+1}$, we know that $\fU^+_{n+1}$ is generated by $e_{i,r}$ where $1 \leq i \leq n$ and $r \in \ZZ$, subject to some quadratic relations. Thus, $\fU^+_{n+1}$ also has a bi-grading which is the same as the bi-grading of $\KHA_{A_n}$. More precisely, let
\begin{equation*}
    \deg(e_{i,r})=(\omega_i,r)
\end{equation*} where $\omega_i$ is the dimension vector with $1$ in vertex $i$ and zero elsewhere, and $r$ is the internal/loop grading. Then, for each $\alpha=(\alpha^1,...,\alpha^n) \in \ZZ^n_{\geq 0}$ and $m \in \ZZ$, $\fU^+_{n+1}$ has the following bi-graded piece
\begin{equation*}
(\fU^+_{n+1})_{\alpha,m}=\Span\bigg\{ e_{i_1,r_1}...e_{i_k,r_{|\alpha|}} \ \bigg| \ \sharp\{a\ |\ i_a=i\}=\alpha^i, \ \sum_{a=1}^{|\alpha|} r_a=m \bigg\} \bigg/  \bigg\langle\text{quadratic relations} \bigg\rangle
\end{equation*} where $|\alpha|=\sum_{i}\alpha^i$.

However, it is standard to see that both the spaces $K(\Rep_{\alpha}Q(A_{n}))_m$ and $(\fU^+_{n+1})_{\alpha,m}$ are infinite-dimensional. For example, consider the case where $n=1$, the dimension vector is $2$ and $m=0$. Then we have
\begin{align*}
    K(\Rep_2(Q(A_1))_0&=\Span\{ \Sym(x^{r_1}_1x^{r_2}_2) \ | \ r_1+r_2=0\} \\
    (\fU^+_{2})_{2,0}&=\Span\{e_{r_1}e_{r_2} \ | \ r_1+r_2=0 \} \bigg/ \bigg\langle\text{quadratic relations} \bigg\rangle
\end{align*} which implies that $r_1=-r_2$. There are infinite many choice of $r_1 \in \ZZ$ and the dimensions are infinite.

The above suggests that we have to restrict ourself to subalgebras of $\KHA_{A_n}$ and $\fU^+_{n+1}$. Thus, we will consider their non-positive or non-negative parts.

\begin{definition} \label{negativesectorKHA}
We define the negative sector, denoted by $\KHA^{\leq 0}_{A_n}$, to be the subalgebra of $\KHA_{A_n}$ that generated by degree-one elements with non-positive powers, i.e. $x^r_{i,j}$ with $r \leq 0$, under the shuffle product. More precisely,
\begin{equation*}
  \KHA^{\leq 0}_{A_n} \coloneqq \bigoplus_{\alpha} \CC[x^{-1}_{i,j}]^{\fS_{\alpha}} \subset \KHA_{A_n}.  
\end{equation*}
\end{definition}

\begin{definition} \label{positivesectorU+n}
We define the positive sector $\fU^{+,\geq 0}_{n+1}$ to be the subalgebra of $\fU^+_{n+1}$ that generated by $e_{i,r}$ with $r \geq 0$ subject to the relations (\ref{er}) whenever it makes sense.
\end{definition}

Then, both $\KHA^{\leq 0}_{A_n}$ and $\fU^{+,\geq 0}_{n+1}$ inherit the bi-grading from $\KHA_{A_n}$ and $\fU^+_{n+1}$ respectively. More precisely, their bigraded pieces are given by 
\begin{align} \label{bigradednegativeKHA_n}
(\KHA^{\leq 0}_{A_n})_{\alpha,m}&=\Span\bigg\{ \Sym\bigg(\prod_{1 \leq i \leq n}\prod_{1 \leq a \leq \alpha^i} x^{-r_{i,a}}_{i,a}\bigg)  \ \bigg| \ r_{i,a} \geq 0 \  \forall \ i, a, \ \text{and} \ \sum r_{i,a}=m \geq 0 \bigg\} \\
(\fU^{+,\geq 0}_{n+1})_{\alpha,m}&=\Span\bigg\{ e_{i_1,r_1}...e_{i_k,r_{|\alpha|}} \ \bigg| \ \sharp\{a\ |\ i_a=i\}=\alpha^i, \ r_a \geq 0, \ \sum_{a=1}^{|\alpha|} r_a=m \bigg\} \bigg/  \bigg\langle\text{quadratic relations} \bigg\rangle \label{bigradedpositivesectorUn+1}
\end{align}

\begin{proposition} \label{comparisionbi-grdeddimension}
We have $\dim(\KHA^{\leq 0}_{A_n})_{\alpha,m}=\dim(\fU^{+,\geq 0}_{n+1})_{\alpha,m}$.
\end{proposition}

\begin{proof}
We prove the case $n=1$ first. When $n=1$, the dimension vector is a non-negative integer $\alpha=d$ and we have the bi-graded pieces are
\begin{align*}
(\KHA^{\leq 0}_{A_1})_{d,m}&=\Span \{ \Sym(x^{-r_1}_1...x^{-r_d}_d) \ | \ r_i \geq 0 \ \forall \ i,\ \text{and} \ \sum r_i=m \geq 0 \},   \\
(\fU^{+,\geq 0}_{2})_{d,m}&=\Span \{ e_{r_1}...e_{r_d} \ | \ r_i \geq 0 \ \forall \ i,\ \text{and} \ \sum r_i=m \geq 0 \}/\langle e_re_s+e_{s-1}e_{r+1} \rangle.
\end{align*}

Note that a symmetrized monoimal is determined (up to scalar) by the multiset $\{r_1,...,r_d\}$. After sorting it, without loss of generality, we can assume $r_1 \geq ... \geq r_d \geq 0$, which means it is a partition of $m$ into at most $d$ parts. So 
\begin{equation*}
\dim(\KHA^{\leq 0}_{A_1})_{d,m}=p_d(m)
\end{equation*} where $p_d(m)$ denote the number of partitions of $m$ into at most $d$ parts. 

Now, we have to show that $\dim(\fU^{+,\geq 0}_{2})_{d,m}=p_m(d)$. In order to show this, we show that the bi-graded space $(\fU^{+,\geq 0}_{2})_{d,m}$ is spanned by 
\begin{equation*}
A=\{ e_{r_1}...e_{r_d} \ | \ r_1 \geq ... \geq r_d \geq 0,\ \text{and} \ \sum r_i=m \geq 0 \},
\end{equation*} i.e., all the other elements can be reduced to elements with weakly decreasing indices by the quadratic relation.

We have to look at the quadratic relation. Observe that when $r<s$ or $s-r \geq 1$, we have the following two cases. First, if $s-r=1$, or equivalently $s=r+1$, then the quadratic relation gives $e_re_{r+1}=-e_re_{r+1}$ which implies that $e_re_{r+1}=0$ for all $r \in \ZZ$. Second, when $s-r \geq 2$, then it gives $e_re_s=-e_{s-1}e_{r+1}$ with $s-1 \geq r+1$. 

Thus, for a general element $e_{r_1}...e_{r_d}$, if it is already in $A$, then we are done. If not, then there exists $1 \leq i \leq d-1$ such that $r_i<r_{i+1}$. Then, we can apply the quadratic relation to get either $e_{r_1}...e_{r_d}=0$ or $e_{r_1}...e_{r_d}=-e_{r_1}...e_{r_{i+1}-1}e_{r_{i}+1}...e_{r_d}$. Continuing this process, $e_{r_1}...e_{r_d}$ will eventually reduced to an element in $A$. 

Clearly, the elements in $A$ are linearly independent. So, $A$ is a basis for $(\fU^{+,\geq 0}_{2})_{d,m}$ and $\dim(\fU^{+,\geq 0}_{2})_{d,m}=p_d(m)$.

Now, we move to the general case. Given a dimension vector $\alpha$ and $m \geq 0$. For the $\KHA$ side, from the condition $\sum_{i=1}^n\sum_{a=1}^{\alpha^i}r_{i,a}=m$ in (\ref{bigradednegativeKHA_n}), we know that $m$ has a composition $m=\sum_{i=1}^nm_i$ where $m_i=\sum_{a=1}^{\alpha^i}r_{i,a}$ for all $1 \leq i \leq n$. Moreover, since the elements are symmetrized monomials, from the study of the $n=1$ case, we can assume $r_{i,1} \geq ... \geq r_{i,\alpha^i} \geq 0$ for all $1 \leq i \leq n$. Thus, the number of the basis elements is given by the sum of number of ways decompose $m$ into compositions $\sum_{i=1}^{n}m_i$ together with partition of each $m_i$ into at most $\alpha^i$ parts, i.e.,
\begin{equation*}
\dim(\KHA^{\leq 0}_{A_n})_{\alpha,m}=\sum_{m=m_1+...+m_n}\prod_{i=1}^n p_{\alpha^i}(m_i).
\end{equation*}

Next, for the 0-affine quantum group $\fU^+_{n+1}$, since we are in the general case and the proof for $n=1$ is done, we can assume $n \geq 2$. Then, we have more quadratic relations, see (\ref{er}). For each word $w=e_{i_1,r_1}....e_{i_d,r_d} \in \fU^+_{n+1}$, we define the \textit{number of inversions} to be 
\begin{equation*}
\Inv(w) \coloneqq \sharp\{(a,b) \ | \ a<b, \ \text{and} \ i_a>i_b \}.
\end{equation*} For example, $\Inv(e_{1,r}e_{2,s})=0$, $\Inv(e_{2,r}e_{1,s})=1$, and $\Inv(e_{2,r_1}e_{3,r_2}e_{1,r_3})=2$. 

The number of inversions induces a filtration $F_0 \subset F_1 \subset ... \subset \fU^+_{n+1}$, where $F_k$ is defined to be
\begin{equation*}
  F_k=\Span\{  w=e_{i_1,r_1}....e_{i_d,r_d} \in \fU^+_{n+1} \ | \ \Inv(w) \leq k\}
\end{equation*} for all $k \geq 0$. Note that when $k=0$, the $F_0$ is spanned by words of the form $e_{1,r_{1,1}}...e_{1,r_{1,d_1}}...e_{n,r_{n,1}}...e_{n,r_{n,d_n}}$.

Taking the associated graded, we have $\gr\fU^+_{n+1}=\bigoplus_{k} F_k/F_{k-1}$. Observe that, the Drinfeld-Serre relation 
\begin{equation*}
e_{i+1,s}e_{i,r}=e_{i,r}e_{i+1,s}-e_{i,r+1}e_{i+1,s-1} 
\end{equation*} in (\ref{er}) shows that every word with at least one number of inversions can be written as a linear combinations of words with with smaller number of inversions. Thus we have $F_k/F_{k-1}=0$ for all $k \geq 1$, and $\gr\fU^+_{n+1}=F_0$. 

By restricting to the subalgebra $\fU^{+,\geq 0}_{n+1}$, we also obtain a filtration $G_0 \subset G_1 \subset ... \subset \fU^{+,\geq 0}_{n+1}$ where $G_k=F_k \cap \fU^{+,\geq 0}_{n+1}$ for all $k \geq 0$. Taking the associated graded, we obtain $\gr\fU^{+,\geq 0}_{n+1}=\bigoplus_{k}G_k/G_{k-1}=G_0=F_0\cap \fU^{+,\geq 0}_{n+1}$ since $F_k/F_{k-1}=0$ for all $k \geq 1$. Then, we have the bi-graded piece is given by
\begin{align*}
 (\gr\fU^{+,\geq 0}_{n+1})_{\alpha,m} &=F_0\cap (\fU^{+,\geq 0}_{n+1})_{\alpha,m}\\
 &=\Span\{  e_{1,r_{1,1}}...e_{1,r_{1,\alpha^1}}...e_{n,r_{n,1}}...e_{n,r_{n,\alpha^n}} \ | \ r_{i,a} \geq 0 \ \forall \ i,a, \ \text{and} \ \sum r_{i,a}=m \}.
\end{align*} Moreover, as we prove the case $n=1$, we can use the quadratic relation $e_{i,r}e_{i,s}=-e_{i,s-1}e_{i,r+1}$ to deduce that $(\gr\fU^{+,\geq 0}_{n+1})_{\alpha,m}$ is spanned by 
\begin{equation*}
 \{  e_{1,r_{1,1}}...e_{1,r_{1,\alpha^1}}...e_{n,r_{n,1}}...e_{n,r_{n,\alpha^n}} \ | \ r_{i,1} \geq ... \geq r_{i,\alpha^i} \geq 0 \ \forall \ i, \ \text{and} \ \sum r_{i,a}=m \}   
\end{equation*} which has the same cardinality as the case in $(\KHA^{\leq 0}_{A_n})_{\alpha,m}$, i.e., $\dim(\gr\fU^{+,\geq 0}_{n+1})_{\alpha,m}=\dim(\KHA^{\leq 0}_{A_n})_{\alpha,m}=\sum_{m=m_1+...+m_n}\prod_{i=1}^n p_{\alpha^i}(m_i).$

Since we prove $\phi_{n+1}$ is a surjection in Theorem \ref{main result 1}, from the definition $\phi_{n+1}(e_{i,r})=x^{-r}_{i,1}$, clearly it restricts to a surjection $\phi_{n+1}|_{ \fU^{+,\geq 0}_{n+1}}:\fU^{+,\geq 0}_{n+1} \twoheadrightarrow \KHA^{\leq 0}_{A_n}$. So
\begin{equation*}
\dim(\KHA^{\leq 0}_{A_n})_{\alpha,m} \leq \dim(\fU^{+,\geq 0}_{n+1} )_{\alpha,m} \leq \dim(\gr\fU^{+,\geq 0}_{n+1})_{\alpha,m}
\end{equation*} and we conclude that $\dim(\KHA^{\leq 0}_{A_n})_{\alpha,m}=\dim(\fU^{+,\geq 0}_{n+1} )_{\alpha,m}$
\end{proof}

From Proposition \ref{phisurjective} and Proposition \ref{comparisionbi-grdeddimension}, we obtain the following algebra isomorphism
\begin{equation*}
    \phi_{n+1}|_{ \fU^{+,\geq 0}_{n+1}}: \fU^{+,\geq 0}_{n+1} \xrightarrow{\simeq} \KHA^{\leq 0}_{A_n}
\end{equation*} which gives isomorphism of subalgebras in $\fU^{+}_{n+1}$ and $\KHA^{\leq 0}_{A_n}$. In order to extend $\phi_{n+1}|_{ \fU^{+,\geq 0}_{n+1}}$ to a full isomorphism $\phi_{n+1}$, we introduce the following degree-shift automorphism
\begin{align*}
    \tau_k&:\fU^+_{n+1}\rightarrow \fU^+_{n+1},\ \tau_k(e_{i,r})=e_{i,r+k}, \ \forall \ i, \ \text{and} \ r,k\in \ZZ, \\
    \eta_k&:\KHA_{A_{n}}\rightarrow \KHA_{A_n},\ \eta_k(x^{r}_{i,1})=x^{r-k}_{i,1},\ \forall \ i, \ \text{and} \ r, k\in \ZZ. 
\end{align*} It is clear to see that $\phi_{n+1} \circ \tau_k=\eta_k  \circ \phi_{n+1}$ for all $k \in \ZZ$.

The following result proves Theorem \ref{main result 1}.

\begin{proposition} \label{injectivityphi}
The map $\phi_{n+1}$ in Theorem \ref{main result 1} is injective.
\end{proposition}

\begin{proof}
We show that $\ker(\phi_{n+1})=0$. Given $x \in \ker(\phi_{n+1})$, we have $x \in \fU^+_{n+1}$ such that $\phi_{n+1}(x)=0$. 

Since $\fU^+_{n+1}$ is bi-graded, i.e., $\fU^+_{n+1}=\bigoplus_{\alpha}\bigoplus_m(\fU^+_{n+1})_{\alpha,m}$, $x$ can be written as a finite sum $x=\sum_{\alpha}\sum_{m}x_{\alpha,m}$ with $x_{\alpha,m} \in (\fU^+_{n+1})_{\alpha,m}$ for some $\alpha \in \ZZ^n_{\geq 0}$, $m \in \ZZ$.

Since $\phi_{n+1}$ is a graded algebra homomorphism, i.e., $\phi_{n+1}((\fU^+_{n+1})_{\alpha,m}) \subset (\KHA_{A_n})_{\alpha,m}$, we have $0=\phi_{n+1}(x)=\sum_{\alpha}\sum_{m}\phi_{n+1}(x_{\alpha,m})$ with each $\phi_{n+1}(x_{\alpha,m})$ in different bi-graded component. Thus, $\phi_{n+1}(x_{\alpha,m})=0$ for all $\alpha$ and $m$. 

It suffices to show that $x_{\alpha,m}=0$ for all $\alpha$ and $m$ such that $x=0$. For each $\alpha$ and $m$, we choose $k$ large enough and apply the degree-shift automorphism so that $\tau_{k}(x_{\alpha,m}) \in \fU^{+,\geq 0}_{n+1}$. Then we have 
\begin{equation*}
0=\eta_k(\phi_{n+1}(x_{\alpha,m}))=\phi_{n+1}|_{ \fU^{+,\geq 0}_{n+1}}(\tau_{k}(x_{\alpha,m}) ).
\end{equation*} Since $\phi_{n+1}|_{ \fU^{+,\geq 0}_{n+1}}$ is injective and $\tau_k$ is an automorphism, we have $\tau_{k}(x_{\alpha,m})=0$ and thus $x_{\alpha,m}=0$. 

\end{proof}

\subsection{Definition of the categorical action}

Roughly speaking, a categorical action of a Kac-Moody Lie algebra $\cg$ consists of a collection of functors and categories that recover actions of Chevalley generators at the level of Grothendieck groups. Under this philosophy, we propose the following definition of categorical action of $\fU^+_n$, or simply called \textit{categorical $\fU^+_n$--action}. 

Recall the notations $\kk$ and $C(n,N)$ in the previous section. Moreover, we denote $\alpha_i=(0,...,-1,1,...,0)$ to be the simple root, where $1 \leq i \leq n-1$ and $-1$ is at the $i$th position.

\begin{definition} \label{catactkha_n}
A \textit{categorical $\fU^+_{n}$-action} consists of a target 2-category $\Kk$, which is triangulated, $\CC$-linear and idempotent complete. The objects in $\Kk$ are
\begin{equation*}
\mathrm{Ob}(\Kk)=\{\Kk(\kk)\ |\ \kk \in C(n,N) \}
\end{equation*} where each $\Kk(\kk)$ is also a triangulated category, and each Hom space $\mathrm{Hom}(\Kk(\kk),\Kk(\Ll))$ is also  triangulated. On those objects $\Kk(\kk)$ we impose the following 1-morphisms:
\begin{equation*}
\bo_{\kk}:\Kk(\kk) \rightarrow \Kk(\kk), \ {\E}_{i,r}\bo_{\kk}=\bo_{\kk-\alpha_{i}}{\E}_{i,r}:\Kk(\kk) \rightarrow \Kk(\kk-\alpha_i),  
\end{equation*}where $r \in \ZZ$. 

On this data, we impose the following condition 
\begin{enumerate}
    \item For each $r \in \ZZ$, ${\E}_{i,r}\bo_{\kk}$ admits left and right adjoint functors $({\E}_{i,r}\bo_{\kk})^L={\E}^L_{i,r}\bo_{\kk-\alpha_i}$, $({\E}_{i,r}\bo_{\kk})^R={\E}^R_{i,r}\bo_{\kk-\alpha_i}$, respectively.
    \item The functors ${\E}_{i,r}\bo_{\kk}$ satisfy 
    \begin{enumerate}
        \item \begin{equation*}
    {\E}_{i,r}{\E}_{i,s}\bo_{\kk} \cong \begin{cases}
    {\E}_{i,s-1}{\E}_{i,r+1}\bo_{\kk}[1] & \text{if} \ r-s \geq 0 \\
    0 & \text{if} \ r-s=-1 \\
    {\E}_{i,s-1}{\E}_{i,r+1}\bo_{\kk}[-1] & \text{if} \ r-s \leq -2. 
    \end{cases}
    \end{equation*}
    \item ${\E}_{i,r}, {\E}_{i+1,s}$ are related by the following exact triangle
    \begin{equation*}
    {\E}_{i,r+1}{\E}_{i+1,s-1}\bo_{\kk} \rightarrow {\E}_{i,r}{\E}_{i+1,s}\bo_{\kk} \rightarrow {\E}_{i+1,s}{\E}_{i,r}\bo_{\kk} .  
    \end{equation*}
    \item 
    \begin{equation*}
    {\E}_{i,r}{\E}_{j,s}\bo_{\kk} \cong {\E}_{j,s}{\E}_{i,r}\bo_{\kk},  \ \text{if} \ |i-j| \geq 2.
    \end{equation*}
    \end{enumerate}
    
    \item There are exact triangles
    \begin{align*}
    &{\E}_{i,r}{\E}^R_{i,r}\bo_{\kk} \rightarrow \bo_{\kk} \rightarrow {\E}^R_{i,r-1}{\E}_{i,r-1}\bo_{\kk} \\
    &{\E}^L_{i,r-1}{\E}_{i,r-1}\bo_{\kk} \rightarrow \bo_{\kk} \rightarrow {\E}_{i,r}{\E}^L_{i,r}\bo_{\kk}
    \end{align*} for all $r \in \ZZ$.
    \item The relations between ${\E}_{i,r}$ and ${\E}^R_{j,s}$, ${\E}^L_{j,s}$ are given by
    \begin{enumerate} 
    \item When $i=j$, there is a so-called shifted condition.
    \begin{align*}
    {\E}_{i,r}{\E}^R_{i,s}\bo_{\kk} &\cong {\E}^R_{i,s-1}{\E}_{i,r-1}\bo_{\kk}[-1], \ \Rm{if} \ 1 \leq r-s \leq k_i+k_{i+1}-1, \\
    {\E}_{i,r}{\E}^L_{i,s}\bo_{\kk} &\cong {\E}^L_{i,s-1}{\E}_{i,r-1}\bo_{\kk}[1], \ \Rm{if} \ -k_i-k_{i+1}+1 \leq r-s \leq -1.
    \end{align*}
    \item When $|i-j|=1$, we have 
    \begin{align*}
        {\E}_{i,r}{\E}^R_{i+1,s}\bo_{\kk} \cong {\E}^R_{i+1,s+1}{\E}_{i,r+1}\bo_{\kk}[1], \ &{\E}_{i,r}{\E}^R_{i-1,s}\bo_{\kk} \cong {\E}^R_{i-1,s}{\E}_{i,r}\bo_{\kk} \\
        {\E}_{i,r}{\E}^L_{i+1,s}\bo_{\kk} \cong {\E}^L_{i+1,s}{\E}_{i,r}\bo_{\kk}, \ &{\E}_{i,r}{\E}^L_{i-1,s}\bo_{\kk} \cong {\E}^L_{i-1,s+1}{\E}_{i,r+1}\bo_{\kk}[-1].
    \end{align*}
    
    \item ${\E}_{i,r}{\E}^R_{j,s}\bo_{\kk} \cong {\E}^R_{j,s}{\E}_{i,r}\bo_{\kk}$, ${\E}_{i,r}{\E}^L_{j,s}\bo_{\kk} \cong {\E}^L_{j,s}{\E}_{i,r}\bo_{\kk}$ for all $|i-j| \geq 2$.
    \end{enumerate}
\end{enumerate}
\end{definition}

We give a few remarks about this definition.

\begin{remark}
The condition (2) can be viewed as a lift/categorification of the relation (\ref{er}) for the generators $e_{i,r}$ in Definition \ref{0affinequantumgroup}.
\end{remark}

\begin{remark}
Denote ${\TTt}_{i,r}\bo_{\kk}={\E}_{i,r}{\E}^R_{i,r}\bo_{\kk}$ and ${\SSs}_{i,r}\bo_{\kk}={\E}^R_{i,r-1}{\E}_{i,r-1}\bo_{\kk} $. Then, by using condition (2)(a), it is standard to see that ${\TTt}_{i,r}\bo_{\kk}$ and ${\SSs}_{i,r}\bo_{\kk}$ are categorical mutually orthogonal, i.e. ${\TTt}_{i,r}{\SSs}_{i,r}\bo_{\kk} \cong 0 \cong {\SSs}_{i,r}{\TTt}_{i,r}\bo_{\kk}$. Moreover, the exact triangles in condition (3) also imply ${\TTt}_{i,r}\bo_{\kk}$ and ${\SSs}_{i,r}\bo_{\kk}$ are idempotents, i.e. ${\TTt}^2_{i,r}\bo_{\kk} \cong {\TTt}_{i,r}\bo_{\kk}$ and ${\SSs}^2_{i,r}\bo_{\kk} \cong {\SSs}_{i,r}\bo_{\kk}$. Such exact triangles are called \textit{idempotent triangles} in \cite{Ho}.
\end{remark}

\begin{remark}
The isomorphisms in Condition (4)(a) really depend on the weight $\kk$, and this is also the shift that comes from the shifted 0-affine algebra in \cite{Hsu1}.
\end{remark}

\subsection{Semiorthogonal decompositions}

Before we state our second result, we need to introduce more notations. Any sequence of non-increasing positive integers $\blam=(\lambda_{1},...,\lambda_{n})$ can be represented as a Young diagram with $n$ rows, aligned on the left, such that the $i$th row has exactly $\lambda_{i}$ cells. For non-negative integers $a, \ b \geq 0$, we denote by $P(a,b)$ the set of Young diagrams $\blam$ such that $\lambda_{1} \leq a$ and $\lambda_{b+1}=0$.

\begin{theorem} \label{mainresult_n}
Given a categorical $\fU^+_{n}$-action $\Kk$. Considering the functors 
\begin{equation*}
{\E}_{1,\blam(1)}...{\E}_{n-1,\blam(n-1)}\bo_{\eta}  
\end{equation*}  where $\eta=(0,...,0,N)$ is the highest weight and $\blam(i) \in P(k_{i+1},\sum_{j=1}^ik_j)$ for all $i$. Then, they satisfy the following properties
\begin{enumerate}
\item  They are all fully-faithful, i.e.
\begin{equation*}
\Hom({\E}_{1,\blam(1)}...{\E}_{n-1,\blam(n-1)}\bo_{\eta},{\E}_{1,\blam(1)}...{\E}_{n-1,\blam(n-1)}\bo_{\eta}) \cong \Hom(\bo_{\eta},\bo_{\eta}),
\end{equation*}
\item They satisfy the semiorthogonal property: \begin{equation*}
\Hom({\E}_{1,\blam(1)}...{\E}_{n-1,\blam(n-1)}\bo_{\eta},{\E}_{1,\blam(1)'}...{\E}_{n-1,\blam(n-1)'}\bo_{\eta}) \cong 0 
\end{equation*}if $(\blam(1),...,\blam(n-1)) <_{pl} (\blam(1)',...,\blam(n-1)')$ where $<_{pl}$ denote the product lexicographical order.
\end{enumerate}    
\end{theorem}

In particular, it is directly to see the following corollary.

\begin{corollary}
Given a categorical $\fU^+_n$-action $\Kk$. We denote 
\begin{equation*}
    \mathrm{Im}{\E}_{1,\blam(1)}...{\E}_{n-1,\blam(n-1)}\bo_{\eta}
\end{equation*} to be the minimal full triangulated subcategory of $\Kk(\kk)$ generated by the class of objects which are essential images of ${\E}_{1,\blam(1)}...{\E}_{n-1,\blam(n-1)}\bo_{\eta}$ with $\blam(i) \in P(k_{i+1},\sum_{j=1}^ik_j)$ for all $i$. Then, we have the following semiorthogonal decomposition
\begin{equation*}
    \Kk(\kk) = \langle \Aa(\kk), \mathrm{Im}{\E}_{1,\blam(1)}...{\E}_{n-1,\blam(n-1)}\bo_{\eta} \rangle_{\blam(i) \in P(k_{i+1},\sum_{j=1}^ik_j)} 
\end{equation*} where $\Aa(\kk)$ is the right orthogonal complement.
\end{corollary}

The idea of the proof is taking the right (or left) adjoint of ${\E}_{1,\blam(1)}...{\E}_{n-1,\blam(n-1)}\bo_{\eta}$ and compute 
\begin{equation} \label{composition_n} 
({\E}_{1,\blam(1)}...{\E}_{n-1,\blam(n-1)}\bo_{\eta})^{R}{\E}_{1,\blam(1)'}...{\E}_{n-1,\blam(n-1)'}\bo_{\eta} \cong {\E}^R_{n-1,\blam(n-1)}...{\E}^R_{1,\blam(1)}{\E}_{1,\blam(1)'}...{\E}_{n-1,\blam(n-1)'}\bo_{\eta}.
\end{equation}

The first step to solve (\ref{composition_n}) is to compute the composition ${\E}^R_{1,\blam(1)}{\E}_{1,\blam(1)'}\bo_{(0,k_1+k_2,...,k_n)}$. Since the two functors in the composition have the same subscript, it suffices to prove Theorem \ref{mainresult_n} in  the $\SL_2$ case.

\subsubsection{The $\SL_2$ case}

We formulate Definition \ref{catactkha_n} in the $n=2$ (or $\SL_2$) version, where we drop the subscript $i$ for simplicity.

\begin{definition} \label{catactkha_2}
A \textit{categorical $\fU^+_{2}$-action} consists of a target 2-category $\Kk$, which is triangulated, $\CC$-linear and idempotent complete. The objects in $\Kk$ are
\begin{equation*}
\mathrm{Ob}(\Kk)=\{\Kk(k,N-k)\ |\ 0 \leq k \leq N \}
\end{equation*} where each $\Kk(k,N-k)$ is also a triangulated category, and each Hom space $\mathrm{Hom}(\Kk(k,N-k),\Kk(l,N-l))$ is also triangulated. On those objects we equip the following 1-morphisms: $\bo_{(k,N-k)}:\Kk(k,N-k) \rightarrow \Kk(k,N-k)$, ${\E}_r\bo_{(k,N-k)}:\Kk(k,N-k) \rightarrow \Kk(k+1,N-k-1)$ where $r \in \ZZ$. 

On this data, we impose the following condition 
\begin{enumerate}
    \item For each $r \in \ZZ$, ${\E}_r\bo_{(k,N-k)}$ admits left and right adjoint functors $({\E}_r\bo_{(k,N-k)})^L$, $({\E}_r\bo_{(k,N-k)})^R$, respectively.
    \item The functors ${\E}_r\bo_{(k,N-k)}$ satisfy 
    \begin{equation*}
    {\E}_{r}{\E}_{s}\bo_{(k,N-k)} \cong \begin{cases}
    {\E}_{s-1}{\E}_{r+1}\bo_{(k,N-k)}[1] & \text{if} \ r-s \geq 0 \\
    0 & \text{if} \ r-s=-1 \\
    {\E}_{s-1}{\E}_{r+1}\bo_{(k,N-k)}[-1] & \text{if} \ r-s \leq -2. 
    \end{cases}
    \end{equation*}
    \item There are exact triangles
    \begin{align*}
    &{\E}_{r}{\E}^R_{r}\bo_{(k,N-k)} \rightarrow \bo_{(k,N-k)} \rightarrow {\E}^R_{r-1}{\E}_{r-1}\bo_{(k,N-k)} \\
    &{\E}^L_{r-1}{\E}_{r-1}\bo_{(k,N-k)} \rightarrow \bo_{(k,N-k)} \rightarrow {\E}_{r}{\E}^L_{r}\bo_{(k,N-k)}
    \end{align*} for all $r \in \ZZ$.
    \item The shifted condition.
    \begin{align*}
    {\E}_{r}{\E}^R_{s}\bo_{(k,N-k)} &\cong {\E}^R_{s-1}{\E}_{r-1}\bo_{(k,N-k)}[-1], \ \Rm{if} \ 1 \leq r-s \leq N-1, \\
    {\E}_{r}{\E}^L_{s}\bo_{(k,N-k)} &\cong {\E}^L_{s-1}{\E}_{r-1}\bo_{(k,N-k)}[1], \ \Rm{if} \ -N+1 \leq r-s \leq -1.
    \end{align*}
\end{enumerate}
\end{definition}

Then, the main result (Theorem \ref{mainresult_n}) for the $\SL_2$ version is the following.

\begin{theorem} \label{mainresult_2}
Given a categorical $\fU^+_{2}$-action $\Kk$. Then the functors $\{{\E}_{\blam}\bo_{(0,N)}\}_{\blam \in P(N-k,k)}$ satisfy the following properties
\begin{enumerate}
\item  Each ${\E}_{\blam}\bo_{(0,N)}$ is fully-faithful, i.e. $\Hom({\E}_{\blam}\bo_{(0,N)},{\E}_{\blam}\bo_{(0,N)}) \cong \Hom(\bo_{(0,N)},\bo_{(0,N)})$,
\item $\Hom({\E}_{\blam}\bo_{(0,N)},{\E}_{\blam'}\bo_{(0,N)}) \cong 0$ 
if $\blam <_{l} \blam'$ (semiorthogonal property).
\end{enumerate}
\end{theorem}

\begin{proof}
We prove the fully-faithfulness first. Given a Young diagram $\blam=(\lambda_1,...,\lambda_k) \in P(N-k,k)$. We apply the right adjoint to get 
\begin{equation*}
 \Hom({\E}_{\blam}\bo_{(0,N)},{\E}_{\blam}\bo_{(0,N)}) \cong \Hom(\bo_{(0,N)},{\E}^{R}_{\blam}{\E}_{\blam}\bo_{(0,N)}). 
\end{equation*} We will show that ${\E}^{R}_{\blam}{\E}_{\blam}\bo_{(0,N)} \cong \bo_{(0,N)}$. Explicitly, we have 
\begin{equation} \label{mainobj}
{\E}^{R}_{\blam}{\E}_{\blam}\bo_{(0,N)}={\E}^R_{\lambda_k}...{\E}^R_{\lambda_1}{\E}_{\lambda_1}...{\E}_{\lambda_k}\bo_{(0,N)}.  
\end{equation} 

Assume $k=1$ first, then ${\E}^{R}_{\blam}{\E}_{\blam}\bo_{(0,N)}={\E}^R_{\lambda_1}{\E}_{\lambda_1}\bo_{(0,N)}$ and there is the following exact triangle from condition (3),
\begin{equation*}
  {\E}_{\lambda_1+1}{\E}^R_{\lambda_1+1}\bo_{(0,N)} \rightarrow  \bo_{(0,N)}  \rightarrow {\E}^R_{\lambda_1}{\E}_{\lambda_1}\bo_{(0,N)}.
\end{equation*} Since $ {\E}_{\lambda_1+1}{\E}^R_{\lambda_1+1}\bo_{(0,N)} \cong 0$, we conclude that ${\E}^R_{\lambda_1}{\E}_{\lambda_1}\bo_{(0,N)} \cong \bo_{(0,N)}$, which is our desired result.

Now, we assume $k \geq 2$. Then, by condition (3) again, we have the following exact triangle
\begin{equation*}
  {\E}_{\lambda_1+1}{\E}^R_{\lambda_1+1}\bo_{(k-1,N-k+1)} \rightarrow  \bo_{(k-1,N-k+1)}  \rightarrow {\E}^R_{\lambda_1}{\E}_{\lambda_1}\bo_{(k-1,N-k+1)}.
\end{equation*} Thus the composition in (\ref{mainobj}) is in the following exact triangle
\begin{equation} \label{1ststepet}
  {\E}^R_{\lambda_k}...{\E}^R_{\lambda_2}{\E}_{\lambda_1+1}{\E}^R_{\lambda_1+1}{\E}_{\lambda_2}...{\E}_{\lambda_k}\bo_{(0,N)} \rightarrow {\E}^R_{\lambda_k}...{\E}^R_{\lambda_2}{\E}_{\lambda_2}...{\E}_{\lambda_k}\bo_{(0,N)} \rightarrow {\E}^{R}_{\blam}{\E}_{\blam}\bo_{(0,N)}.
\end{equation}

Note that since $\blam \in P(k,N-k)$, we have $0 \leq \lambda_k \leq ... \leq \lambda_1 \leq N-k$. This implies that $0 \leq \lambda_1-\lambda_i \leq N-k$ for all $i \geq 2$. By using condition (4) repeatedly, we get
\begin{align*}
 &{\E}^R_{\lambda_k}...{\E}^R_{\lambda_2}{\E}_{\lambda_1+1}\bo_{(k-2,N-k+2)} \\
 &\cong {\E}^R_{\lambda_k}...{\E}_{\lambda_1+2}{\E}^R_{\lambda_2+1}\bo_{(k-2,N-k+2)}[1] \ (\text{since} \ 1 \leq \lambda_1-\lambda_2+1 \leq N-k+1 \leq N-1) \\
 &\cong ... \\
 &\cong {\E}^R_{\lambda_k}{\E}_{\lambda_1+k-1}...{\E}^R_{\lambda_2+1}\bo_{(k-2,N-k+2)}[k-2]\ (\text{since} \ k-2 \leq \lambda_1-\lambda_{k-1}+k-2 \leq N-k+k-2 \leq N-1) \\
 & \cong {\E}_{\lambda_1+k}{\E}^R_{\lambda_k+1}...{\E}^R_{\lambda_2+1}\bo_{(k-2,N-k+2)}[k-1] \ (\text{since} \ k-1 \leq \lambda_1-\lambda_k+k-1 \leq N-k+k-1 \leq N-1).
\end{align*} Thus,
\begin{equation*}
  {\E}^R_{\lambda_k}...{\E}^R_{\lambda_2}{\E}_{\lambda_1+1}{\E}^R_{\lambda_1+1}{\E}_{\lambda_2}...{\E}_{\lambda_k}\bo_{(0,N)} \cong  {\E}_{\lambda_1+k}{\E}^R_{\lambda_k+1}...{\E}^R_{\lambda_2+1}{\E}^R_{\lambda_1+1}{\E}_{\lambda_2}...{\E}_{\lambda_k}\bo_{(0,N)} [k-1] \cong 0
\end{equation*} and it implies 
\begin{equation*}
    {\E}^{R}_{\blam}{\E}_{\blam}\bo_{(0,N)} \cong {\E}^R_{\lambda_k}...{\E}^R_{\lambda_2}{\E}_{\lambda_2}...{\E}_{\lambda_k}\bo_{(0,N)}.
\end{equation*} from the exact triangle (\ref{1ststepet})

Continuing this process with the same argument, we get
\begin{equation*}
 {\E}^{R}_{\blam}{\E}_{\blam}\bo_{(0,N)} \cong {\E}^R_{\lambda_k}...{\E}^R_{\lambda_2}{\E}_{\lambda_2}...{\E}_{\lambda_k}\bo_{(0,N)} \cong ... \cong   {\E}^R_{\lambda_k}{\E}_{\lambda_k}\bo_{(0,N)} \cong \bo_{(0,N)}.
\end{equation*} Hence, we prove the fully-faithfulness.

Now, we prove the semiorthogonal property. Given $\blam,\ \blam' \in P(N-k,k)$ such that $\blam <_l \blam'$. Then, by taking the right adjoint, we have 
\begin{equation*}
 \Hom({\E}_{\blam}\bo_{(0,N)},{\E}_{\blam'}\bo_{(0,N)}) \cong \Hom(\bo_{(0,N)},{\E}^R_{\blam}{\E}_{\blam'}\bo_{(0,N)}).
\end{equation*} We will show that ${\E}^R_{\blam}{\E}_{\blam'}\bo_{(0,N)} \cong 0$. Explicitly, we have 
\begin{equation} \label{2ndmainobj}
{\E}^{R}_{\blam}{\E}_{\blam'}\bo_{(0,N)}={\E}^R_{\lambda_k}...{\E}^R_{\lambda_1}{\E}_{\lambda'_1}...{\E}_{\lambda'_k}\bo_{(0,N)}.  
\end{equation} 

Since $\blam <_l \blam'$, we define $i \coloneqq \text{min}\{1 \leq j \leq k \ | \ \lambda_j<\lambda'_j \}$ such that $\lambda_a=\lambda'_a$ for all $1 \leq a \leq i-1$. Then, like the argument in the proof of the fully-faithfulness property, (\ref{2ndmainobj}) becomes
\begin{equation*}
 {\E}^R_{\lambda_k}...{\E}^R_{\lambda_1}{\E}_{\lambda'_1}...{\E}_{\lambda'_k}\bo_{(0,N)} \cong {\E}^R_{\lambda_k}...{\E}^R_{\lambda_2}{\E}_{\lambda'_2}...{\E}_{\lambda'_k}\bo_{(0,N)} 
 \cong... \cong {\E}^R_{\lambda_k}...{\E}^R_{\lambda_i}{\E}_{\lambda'_i}...{\E}_{\lambda'_k}\bo_{(0,N)}.   
\end{equation*}

Now we have $0 \leq \lambda_i<\lambda'_i \leq N-k$, thus $1 \leq \lambda'_i-\lambda_i \leq N-k \leq N-1$. By condition (4), we obtain
\begin{equation} \label{2ndmainobj'}
 {\E}^R_{\lambda_k}...{\E}^R_{\lambda_i}{\E}_{\lambda'_i}...{\E}_{\lambda'_k}\bo_{(0,N)} \cong {\E}^R_{\lambda_k}...{\E}_{\lambda'_i+1}{\E}^R_{\lambda_i+1}...{\E}_{\lambda'_k}\bo_{(0,N)} [1].  
\end{equation} Next, since $0 \leq \lambda_{i+1} \leq \lambda_i<\lambda'_i \leq N-k$, we also have $2 \leq \lambda'_i - \lambda_{i+1}+1 \leq N-k+1 \leq N-1$ (since $2 \leq i+1 \leq k$). By condition (4) again, 
\begin{equation*}
{\E}^R_{\lambda_k}...{\E}^R_{\lambda_{i+1}}{\E}_{\lambda'_i+1}{\E}^R_{\lambda_i+1}...{\E}_{\lambda'_k}\bo_{(0,N)} [1] \cong  {\E}^R_{\lambda_k}...{\E}_{\lambda'_i+2}{\E}^R_{\lambda_{i+1}+1}{\E}^R_{\lambda_i+1}...{\E}_{\lambda'_k}\bo_{(0,N)} [2].   
\end{equation*} Continuing this process, (\ref{2ndmainobj'}) becomes 
\begin{align*}
 {\E}^R_{\lambda_k}...{\E}^R_{\lambda_i}{\E}_{\lambda'_i}...{\E}_{\lambda'_k}\bo_{(0,N)} &\cong {\E}^R_{\lambda_k}...{\E}_{\lambda'_i+1}{\E}^R_{\lambda_i+1}...{\E}_{\lambda'_k}\bo_{(0,N)} [1] \\ 
 &\cong ... \\
 &\cong  {\E}_{\lambda'_i+k-i+1}{\E}^R_{\lambda_k+1}...{\E}^R_{\lambda_i+1}{\E}_{\lambda'_{i+1}}...{\E}_{\lambda'_k}\bo_{(0,N)} [k-i+1] \cong 0
\end{align*} which shows the semiorthogonal property. 
\end{proof}

Finally, since we prove the $\SL_2$ case, the proof of Theorem \ref{mainresult_n} is the same as the argument in the proof of \cite[Theorem 4.1]{Hsu2}.

\subsection{Main example}

In this subsection, we provide an example where there is a categorical $\fU^+_{n}$-action.

For each $\kk \in C(n,N)$, recall the $\kk$-partial flag variety $\Fl_{\kk}(\CC^N)$ in (\ref{eq fl}). On $\Fl_{\kk}(\CC^N)$, we denote $\V_i$ to be the tautological bundle of rank $\sum_{j=1}^ik_j$ for all $1 \leq i \leq n-1$. There is a natural correspondence relates $\Fl_{\kk}(\CC^N)$ and $\Fl_{\kk-\alpha_i}(\CC^N)$
\begin{equation*}
\xymatrix{ 
&&W^{1}_{i}(\kk)
\ar[ld]_{p_1} \ar[rd]^{p_2}   \\
& \Fl_{\kk}(\CC^N) && \Fl_{\kk-\alpha_{i}}(\CC^N)
}
\end{equation*}  where 
$W^{1}_{i}(\kk) \coloneqq \{(V_{\bullet},V_{\bullet}') \in  \Fl_{\kk}(\CC^N) \times \Fl_{\kk-\alpha_i}(\CC^N) \ | \ V_{j}=V'_{j} \ \Rm{for} \ j \neq i,  \ V_{i} \subset V'_{i}\}$, and $p_{1}$, $p_{2}$ are the natural projections.

We define the functors ${\E}_{i,r}\bo_{\kk}$ by the following
\begin{equation} \label{eir}
{\E}_{i,r}\bo_{\kk} \coloneqq p_{2*}(p_{1}^{*} \otimes (\V'_{i}/\V_{i})^{\otimes r}):\D^b\Coh(\Fl_{\kk}(\CC^N)) \rightarrow \D^b\Coh(\Fl_{\kk-\alpha_{i}}(\CC^N))
\end{equation} for all $r \in \ZZ$.

Then, we have the following result, which follows from the works in \cite{Hsu1} and \cite{Hsu3}.

\begin{theorem} \label{main result 3}
Let $\Kk$ be the triangulated 2-category whose nonzero objects are $\Kk(\kk)=\D^b\Coh(\Fl_{\kk}(\CC^N))$ where $\kk \in C(n,N)$. The 1-morphisms are ${\E}_{i,r}\bo_{\kk}$ that defined in (\ref{eir}) and their compositions. The two morphisms are maps between kernels. Then, the above data gives a categorical $\fU^+_{n}$-action in Definition \ref{composition_n}.
\end{theorem}

\begin{proof}

Observe that our definition of the functors ${\E}_{i,r}\bo_{\kk}$ in (\ref{eir}) is precisely the definition of the functors ${\F}_{i,r}\bo_{\kk}$ in \cite[Definition 5.5]{Hsu1}. Thus, we will use the main result in \cite{Hsu1}, i.e. there is a categorical action of the shifted 0-affine algebra on the derived categories of coherent sheaves on partial flag varieties, to obtain conditions (1) to (4) in Definition \ref{catactkha_n}. 

Note that, in the presentation of the shifted 0-affine algebra
\cite[Definition 2.6]{Hsu1}, we only require finite generators and relations. In particular, the loop generators  $f_{i,r}1_{\kk}$ with $s$ within a certain range ($0 \leq r \leq k_{i+1}$). So, in the categorical action of shifted 0-affine algebra \cite[Definition 3.1]{Hsu1}, we only have functors ${\F}_{i,r}\bo_{\kk}$ for $0 \leq r \leq k_{i+1}$. 

To prove the result, we need to introduce the following functors to help us
\begin{equation*}
{\spi}^-_i\bo_{\kk}\coloneqq \otimes \det(\V_i/\V_{i-1})^{-1}[1-k_{i}]:\D^b\Coh(\Fl_{\kk}(\CC^N)) \rightarrow \D^b\Coh(\Fl_{\kk}(\CC^N))
\end{equation*} for all $1 \leq i \leq n-1$. Then, we have
\begin{equation} \label{twist}
{\F}_{i,r}\bo_{\kk} \cong ({\spi}^-_i)^{-r}{\F}_{i,0}({\spi}^-_i)^{r}\bo_{\kk}[-r]    
\end{equation} for all $r \in \ZZ$. 

Since there are also functors with the same notations ${\E}_{i,r}\bo_{\kk}$ in the opposite direction from \cite{Hsu1}, to distinguish them, we will add an extra prime in all the notations, i.e., we prove the this result in terms of the functors ${\F}'_{i,r}\bo_{\kk}$.

For condition (1), since ${\F}'_{i,0}\bo_{\kk}$ admits left and right adjoints from \cite[Theorem 5.6]{Hsu1} and ${\spi}^+_i\bo_{\kk}$ is invertible, we have 
\begin{equation*}
   {\F}'^R_{i,r}\bo_{\kk} \cong \biggl( ({\spi}^+_i)^{-r}{\F}'_{i,0}({\spi}^+_i)^{r}\bo_{\kk}[r] \biggr)^{R} \cong ({\spi}^+_i)^{-r}{\F}'^R_{i,0}({\spi}^+_i)^{r}\bo_{\kk-\alpha_i}[-r], 
\end{equation*} and similarly for the left adjoint.

For condition (2), it follows from the same argument in the proof of \cite[Lemma 5.9 (1)]{Hsu1} and \cite[Lemma 5.12]{Hsu1}.

For condition (3), it suffices to prove the exact triangle for the right adjoint case, and the left adjoint case will follows by taking the left adjoint. From \cite[Lemma 5.10]{Hsu3}, we have the following exact triangle
\begin{equation} \label{etr=0}
    {\F}'_{i,0}{\F}'^{R}_{i,0}\bo_{\kk} \rightarrow \bo_{\kk} \rightarrow   {\F}'^R_{i,-1}{\F}'_{i,-1}\bo_{\kk}.  
\end{equation} Using the twist by ${\spi}^+_i\bo_{\kk}$, we have 
\begin{align*}
   ({\spi}^+_i)^{-r}{\F}'_{i,0}{\F}'^{R}_{i,0}({\spi}^+_i)^{r}\bo_{\kk} &\cong   ({\spi}^+_i)^{-r}{\F}'_{i,0}({\spi}^+_i)^{r}[r] ({\spi}^+_i)^{-r}{\F}'^{R}_{i,0}({\spi}^+_i)^{r}[-r]\bo_{\kk} \cong {\F}'_{i,r}{\F}'^{R}_{i,r}\bo_{\kk}, \\
   ({\spi}^+_i)^{-r}{\F}'^R_{i,-1}{\F}'_{i,-1}({\spi}^+_i)^{r}\bo_{\kk} &\cong   ({\spi}^+_i)^{-r}{\F}'^R_{i,-1}({\spi}^+_i)^{r}[-r] ({\spi}^+_i)^{-r}{\F}'_{i,-1}({\spi}^+_i)^{r}[r]\bo_{\kk} \cong {\F}'^R_{i,r-1}{\F}'_{i,r-1}\bo_{\kk}.
\end{align*} Thus, by conjugation with the exact triangle (\ref{etr=0}) with $({\spi}^+_i)^{-r}\bo_{\kk}$, we obtain 
\begin{equation*}
    {\F}'_{i,r}{\F}'^{R}_{i,r}\bo_{\kk} \rightarrow \bo_{\kk} \rightarrow   {\F}'^R_{i,r-1}{\F}'_{i,r-1}\bo_{\kk}
\end{equation*} which is the desired result. 

Finally, we prove condition (4). Like condition (3),  it suffices to prove the right adjoint case and the rest follows by taking left adjoint. We need to use the functors ${\E}'_{i,r}\bo_{\kk}$ in the opposite direction from \cite{Hsu1}. Since we will keep using conditions from the categorical action of shifted 0-affine algebra (\cite[Definition 3.1]{Hsu1}), when we say conditions we always means conditions from \cite[Definition 3.1]{Hsu1}.

For (a), we compute 
\begin{align*}
 {\F}'_{i,r}{\F}'^{R}_{i,s}\bo_{\kk} &={\F}'_{i,r}\bo_{\kk+\alpha_i}({\F}'_{i,s}\bo_{\kk+\alpha_i})^R \cong  {\F}'_{i,r}({\spi}_i^-)^{-s}{\E}'_{i,-k_i}({\spi}_i^-)^{s-1}\bo_{\kk}[s] \ (\text{by condition (3)(c)}) \\
 & \cong {\F}'_{i,r}{\E}'_{i,-k_i-s}({\spi}_i^-)^{-1}\bo_{\kk} \ (\text{by condition (7)(a)}) \\
\end{align*} Since the assumption $1 \leq r-s \leq k_i+k_{i+1}-1$ is equivalent to $1-k_i \leq r-s-k_i \leq k_{i+1}-1$, using condition (10)(c) in \cite[Definition 3.1]{Hsu1} we obtain 
\begin{align*}
  {\F}'_{i,r}{\E}'_{i,-k_i-s}({\spi}_i^-)^{-1}\bo_{\kk} &\cong   {\E}'_{i,-k_i-s}{\F}'_{i,r}({\spi}_i^-)^{-1}\bo_{\kk} \cong {\E}'_{i,-k_i-s}({\spi}_i^-)^{-1}{\F}'_{i,r-1}\bo_{\kk}[-1] \\
  & \cong ({\spi}_i^-)^{-s+1}{\E}'_{i,-k_i-1}({\spi}_i^-)^{s-2}
  \bo_{\kk-\alpha_i}[s-1]{\F}'_{i,r-1}\bo_{\kk}[-1] \ (\text{by condition (7)(a)}) \\
  & \cong {\F}'^{R}_{i,s-1}{\F}'_{i,r-1}\bo_{\kk}[-1] \ (\text{by condition (3)(c)})
\end{align*} which proves (a).

For (b), since the arguments are similar, we only prove one of them, say the non-trivial one. We compute 
\begin{align*}
    {\F}'_{i,r}{\F}'^{R}_{i+1,s}\bo_{\kk} &\cong {\F}'_{i,r}({\spi}_{i+1}^-)^{-s}{\E}'_{i+1,-k_{i+1}}({\spi}_{i+1}^-)^{s-1}\bo_{\kk}[s] \ (\text{by condition (3)(c)}) \\
    & \cong ({\spi}_{i+1}^-)^{-s}{\F}'_{i,r+s}{\E}'_{i+1,-k_{i+1}}({\spi}_{i+1}^-)^{s-1}\bo_{\kk}[2s] \ (\text{by condition (8)(b)}) \\
    & \cong ({\spi}_{i+1}^-)^{-s}{\E}'_{i+1,-k_{i+1}}{\F}'_{i,r+s}({\spi}_{i+1}^-)^{s-1}\bo_{\kk}[2s] \ (\text{by condition (9)}) \\
     & \cong ({\spi}_{i+1}^-)^{-s}{\E}'_{i+1,-k_{i+1}}({\spi}_{i+1}^-)^{s-1}\bo_{\kk-\alpha_i}{\F}'_{i,r+1}\bo_{\kk}[s+1] \ (\text{by condition (8)(b)}) \\
     & \cong ({\spi}_{i+1}^-)^{-s-1}{\E}'_{i+1,-k_{i+1}+1}({\spi}_{i+1}^-)^{s}\bo_{\kk-\alpha_i}{\F}'_{i,r+1}\bo_{\kk}[s+2] \ (\text{by condition (7)(a)}) \\
     & \cong ({\F}'_{i+1,s+1}\bo_{\kk-\alpha_i+\alpha_{i+1}})^R{\F}'_{i,r+1}\bo_{\kk}[1] \ (\text{by condition (3)(c)}) \\
     &\cong {\F}'^{R}_{i+1,s+1}{\F}'_{i,r+1}\bo_{\kk}[1]
\end{align*} which proves (b).

For (c), the proof also follows from conditions in \cite[Definition 3.1]{Hsu1} and we leave it to the readers.




\end{proof}

Finally, by Theorem \ref{catactkha_n}, the categorical $\fU^+_{n}$-action on $\bigoplus_{\kk} \D^b\Coh(\Fl_{\kk}(\CC^N))$ gives semiorthogonal decomposition on $\D^b\Coh(\Fl_{\kk}(\CC^N))$ for each $\kk$. As we mention in the introduction in \cite{Hsu2}, this semiorthogonal decomposition recovers the exceptional collection given by Kapranov \cite{Ka1}, \cite{Ka2}.

\end{document}